\documentclass[11pt]{article}
\usepackage{amsmath, latexsym, amsfonts, amssymb, amsthm, amscd}
\usepackage[left=2.5cm, right=2.5cm, top=2.5cm, bottom=2.5cm]{geometry}
\usepackage{upquote}
\usepackage{color}
\usepackage{setspace}
\usepackage[labelfont=bf]{caption}

\usepackage[utf8]{inputenc}
\usepackage[T1]{fontenc}
\usepackage[english]{babel}
\usepackage{dsfont}
\usepackage{mathtools}

\usepackage{xcolor}
\definecolor{Maroon}{HTML}{ad2231}
\definecolor{webgreen}{HTML}{008000}

\usepackage{hyperref}
\hypersetup{colorlinks, breaklinks, urlcolor=Maroon, linkcolor=Maroon, citecolor=webgreen} 

\allowdisplaybreaks

\usepackage{bookmark}



\newtheorem{corollary}{Corollary}
\newtheorem{proposition}{Proposition}
\newtheorem{lemma}{Lemma}
\newtheorem{remark}{Remark}

\newtheorem{theorem}{Theorem}
\newtheorem{definition}{Definition}
\newtheorem{example}{Example}

\theoremstyle{definition}
\newtheorem{Assumption}{Assumption}

\begin{document}
\title{The existence of a giant cluster for percolation on large Crump-Mode-Jagers trees}
\author{G. Berzunza\footnote{ {\sc Department of Mathematics, Uppsala University. L\"agerhyddsv\"agen 1, Hus 1, 6 och 7, Box 480, 751 06 Uppsala, Sweden.} E-mail: gabriel.berzunza-ojeda@math.uu.se}}
\maketitle

\vspace{0.1in}

\begin{abstract} 
In this paper, we consider random trees associated with the genealogy of Crump-Mode-Jagers processes and perform Bernoulli bond-percolation whose parameter depends on the size of the tree. Our purpose is to show the existence of a giant percolation cluster for appropriate regimes as the size grows. We stress that the family trees of Crump-Mode-Jagers processes include random recursive trees, preferential attachment trees, binary search trees for which this question has been answered by Bertoin \cite{Be2013}, as well as (more general) $m$-ary search trees, fragmentation trees, median-of-($2\ell+1$) binary search trees, to name a few, where up to our knowledge percolation has not been studied yet. 

\end{abstract}


\noindent {\sc Subject Classes}: 60J80; 60K35; 05C05.

\section{Introduction and main results}


Consider a graph $G_{n}$ of large but finite size $n \in \mathbb{N}$ and perform Bernoulli bond-percolation with parameter $p_{n} \in (0,1)$ that depends on the size of $G_{n}$ (typically the size of a graph refers to its number of vertices but not necessarily). This means we first pick a finite graph and then remove each edge with probability $1 - p_{n}$, independently of the other edges, inducing a partition of its set of vertices into connected clusters. A natural problem in this setting is to show the existence of a giant cluster for appropriate regimes of the percolation parameter $p_{n}$, when the size of the graph grows. More precisely, one is interested in finding a supercritical $p_{n}$ such that, with high probability as $n \rightarrow \infty$, there exists a cluster that is of a size comparable to the entire graph. Let us recall some known answers to this question in some important instances.

In the case of the complete graph with $n$ vertices, a classical result due to Erd\"os and R\'enyi (see for instance \cite{Bollo2001}) shows that for\footnote{For two sequence of real numbers $(a_{n})_{n \geq 1}$ and $(b_{n})_{n \geq 1}$, we write $a_{n} \sim b_{n}$ if $a_{n}/b_{n} \rightarrow 1$ as $n \rightarrow \infty$, and we write $a_{n} \ll b_{n}$ or $b_{n} \gg a_{n}$ if and only if $a_{n}/b_{n} \rightarrow 0$ as $n \rightarrow \infty$. } $p_{n} \sim c/n$ as $n \rightarrow \infty$ with $c > 1$ fixed, with high probability, there is an unique giant cluster of size close to $\theta(c) n$ where $\theta(c)$ is the unique strictly positive solution to the equation $x + e^{-cx} = 1$. Second, consider an uniform Cayley tree with $n$ vertices (i.e.\ a tree picked uniformly at random amongst the $n^{n-2}$ trees on a set of $n$ labelled vertices). Pitman \cite{Pit1999, Pit2006} showed that for $1-p_{n} \sim c/\sqrt{n}$ as $n \rightarrow \infty$ with a fixed $c >0$, the sequence of sizes of the percolation clusters ranked in decreasing order and renormalized by a factor $1/n$ converges weakly as $n \rightarrow \infty$ to a random mass partition which can be described explicitly in terms of a conditioned Poisson measure. Finally, Bertoin \cite{Be2013} has shown that for fairly general families of trees with $n$ vertices, the supercritical regime corresponds to percolation parameters of the form $1-p_{n} \sim  c/ \ell(n)$ as $n \rightarrow \infty$, where $c >0$ fixed and $\ell(n)$ is an estimate of the height of a typical vertex in the tree structure. Roughly speaking, Bertoin \cite{Be2013} established that under the previous regime the size  of the cluster containing the root is of order $n$ as $n \rightarrow \infty$. The latter result includes for instance some important families of random trees, such as random recursive trees, preferential attachment trees, binary search trees, etc, where it is well-known that $\ell(n) = \ln n$; see \cite{Drmote2009}, \cite[Section 4.4]{Durret2010}. 

The main purpose of this work is to investigate the same question for large random Crump-Mode-Jagers trees or CMJ-trees for short. More precisely, CMJ-trees are the family trees (or genealogical trees) of Crump-Mode-Jagers processes also referred to as general, or age-dependent branching processes; we refer for further details the classic book of Jagers \cite{Ja}. These are general branching population models where the number of individuals can be measured or counted in many different ways: those born, those alive or in some sub-phase of life, for instance. More generally, one can assign random characteristics or weights to each of the individuals and measure the size of the population according to those characteristics (for instance special choices of reproduction point process and counting yield the classical Galton-Watson or Bellman-Harris processes). We postpone its formal definition for later in this work and continue by informally describing our main results. Loosely speaking, we study Bernoulli bond-percolation on CMJ-trees at the time when the total weight (``size'') of the underlying CMJ-process reaches $n$ in three different regimes,
\begin{itemize}
\item weakly supercritical, $\frac{1}{\ln n} \ll 1- p_{n} \ll 1$,
\item supercritical, $1 - p_{n} \sim \frac{c}{\ln n}$ for some $c > 0$ fixed, and
\item strongly supercritical, $0 < 1- p_{n} \ll \frac{1}{\ln n}$.
\end{itemize}

\noindent We show that under standard conditions on the underlying CMJ-process that the root cluster is of order $n^\frac{\kappa(p_{n})}{\alpha}$, where $\kappa(p_{n}) >0$ is a function of the percolation parameter and $\alpha > 0$ is the so-called {\sl Malthusian parameter}. We have used the same terminology as in \cite{Er2016}, where only random recursive trees are studied. In Section \ref{applications}, we shall see that several important families of random trees can be constructed as family trees of a CMJ-process stopped at a suitable time. For example, random recursive trees, preferential attachment trees and binary search trees
where the existence of a giant cluster has been shown by Bertoin \cite{Be2013}. On the other hand, the general nature of the CMJ-processes will allow us to provide new results on percolation for (more general) $m$-ary search trees \cite{Mun1971}, fragmentation trees \cite{Svante2008}, median-of-($2\ell+1$) binary search trees \cite{Dev1993} and so-called splitting trees introduced in \cite{Gei1996}, to name a few. 

In the rest of the introduction, we are going to describe our setting more precisely and give the exact definition of CMJ-trees. This will enable us to state our main result in Section \ref{Main}.

\subsection{Crump-Mode-Jagers trees} \label{Model}

We start by recalling the definition of Crump-Mode-Jagers processes (CMJ-processes) whose associated family trees call CMJ-trees. Following Jagers \cite{Ja}, we present a CMJ-process as a general branching process that starts with a single individual born at time $0$.  We use the usual Ulam-Harris notation and introduce the set of labels, $\mathbb{U} = \bigcup_{n=0}^{\infty} \mathbb{N}^{n}$, with the convention $\mathbb{N}^{0} = \{ \varnothing \}$. The ancestor has label $\varnothing$. An individual with label $u = (u_{1}, \dots, u_{n}) \in \mathbb{U}$ belongs to the $n$-th generation and it is understood to be the $u_{n}$-th descendant of $(u_{1}, \dots, u_{n-1})$, which is the $u_{n-1}$-th descendant of $(u_{1}, \dots, u_{n-2})$ and so on. The initial individual has a random number $N$ of children, born at random times $(\xi_{i})_{i=1}^{N}$ where $0 \leq N \leq \infty$ and $0 \leq \xi_{1} \leq \xi_{2} \leq \cdots \leq \xi_{N}$. Formally, we describe the birth times $(\xi_{i})_{i=1}^{N}$ as a point process $\Xi$ on $[0, \infty)$, i.e., $\Xi = \sum_{i=1}^{N} \delta_{\xi_{i}}$ is an integer-valued random measure, where $\delta_{t}$ is a point mass (Dirac measure) at
time $t \geq 0$; see e.g.\ \cite{Kall2002}.  
We denote by $\mu(\cdot) := \mathbb{E}[\Xi(\cdot)]$ the intensity measure of $\Xi$, and write $\mu(t) := \mu([0, t]) = \mathbb{E}[\Xi([0,t])]$. In particular, we have $N = \Xi([0,\infty))$, and thus, $\mu(\infty)= \mathbb{E}[N]$. Every child that is born evolves in the same way, i.e., every individual $u$ has its own copy $\Xi_{u}$ of $\Xi$ (where now $\xi_{i}$ means the age of the mother when child $i$ is born); these copies are assumed to be independent and identically distributed. We denote the time an individual $u$ is born by $\sigma_{u}$. We also assume that each individual has a random lifetime $\lambda \in [0, \infty]$ (for several of our applications we assume $\lambda \equiv \infty$)\footnote{Formally, one assigns to each possible individual $u$ a copy $(\Omega_{u}, \mathcal{F}_{u}, \nu_{u})$ of some generic probability space $(\Omega, \mathcal{F}, \nu)$ on which we define $\Xi$, and possibly other random characteristic or weight $\phi$. The general branching process is then defined on the product $\prod_{u}(\Omega_{u}, \mathcal{F}_{u}, \nu_{u})$ of these probability spaces.}.

The simplest way to measure or monitor the evolution of the CMJ-process is to consider the process $Z = (Z(t), t \geq 0)$ of the total number  of individuals that have been born up to time $t \geq 0$, i.e.\ the number of births in $[0,t]$. More precisely,
\begin{eqnarray*}
 Z(t) = \sum_{u} \mathbf{1}_{\{ \sigma_{u} \leq t \}}, \hspace*{6mm} t \geq 0;
\end{eqnarray*}

\noindent see e.g.\ \cite{Ja, JaNe}. Following Jagers' work on CMJ-processes (see e.g.\ \cite{Ja, JaNe, Ner1981, Ner1984}), it is going to be relevant to monitor the evolution of individuals that satisfy some random property, instead of the total number of births in some fixed time interval. This random property or characteristic of an individual might be unrelated or heavily dependent on its reproduction behaviour. More precisely, a characteristic or weight of an individual is a random function $\phi: \mathbb{R}_{+} \rightarrow \mathbb{R}_{+}$ that assigns the value $\phi(t)$ when the individual's age is $t \geq 0$. We assume that $\phi$ is c\`adl\`ag (we may extend $\phi$ to $\mathbb{R}$ by setting $\phi(t) = 0$ for $t < 0$). We assume that each individual $u$ has its own copy $\phi_{u}$ and thus we associated to each of them a triple $(\Xi_{u}, \lambda_{u}, \phi_{u})$. These triples for all individuals are independent and identically distributed. We then define the $\phi$-counted process $Z^{\phi} = (Z^{\phi}(t), t \geq 0)$ by
\begin{eqnarray*}
Z^{\phi}(t):= \sum_{u: \sigma_{u} \leq t} \phi_{u}(t- \sigma_{u}), \hspace*{5mm} t \geq 0,
\end{eqnarray*}

\noindent and say that $Z^{\phi}_{t}$ is the total weight at time $t$ of all individuals that have been born so far (recall that $u$ is born at time $\sigma_{u}$, and thus has age $t-\sigma_{u}$ at time $t $). If $\phi \equiv 1$, we have that $Z^{\phi}= Z$. On the other hand, the characteristic $\phi =  \mathds{1}_{[0, \lambda)}$ yields to the number $
Z^{\phi}(t) = \sum_{u} \mathds{1}_{\{ \sigma_{u} \leq t < \sigma_{u} + \lambda_{u} \}}$ of individuals alive at time $t \geq 0$.

Following \cite{Holm2017}, we let $T(\infty)$ be the family tree of the entire CMJ-process or (complete) CMJ-tree. This tree is obtained from the general branching process described at the beginning of this section by ignoring the time structure. Specifically, the individuals in the population are seen as vertices where the initial individual is the root. The children of a vertex in the tree are the same as the children in the general branching process. The tree $T(\infty)$ may be infinite which happens when the process does not die out, i.e.\ $Z(\infty) = \infty$. For $t \geq 0$, we let $T(t)$ be the CMJ-tree consisting of all individuals born up to time $t$. Note that the number of vertices at time $t \geq 0$ is given by  $Z(t)$. Clearly, $T(t)$ is an unordered tree for $t >0$. However,  one could get an ordered tree by adding an additional ordering of the children of each individual. This can be done by taking the children in order of birth, or by choosing a random order; we refer to \cite[Remark 5.1]{Holm2017} for further details. Finally, observe that the random tree $T(t)$ has a random size for $t > 0$ (possibly infinite). In this work, we shall be mainly interested in CMJ-trees with a given number of vertices or when some random property is fulfilled. More precisely, we have the following definition.

\begin{definition} \label{def1}
Fix a random characteristic or weight $\phi$. For $n \in \mathbb{N}$, we let
\begin{eqnarray*}
\tau^{\phi} (n):= \inf \{t \geq 0: Z^{\phi}(t) \geq n \}
\end{eqnarray*}

\noindent be the first time the total weight is at least $n$ (as usual $\inf \varnothing = \infty$). We exclude the case $\phi \equiv 0$ which would yield to $\tau^{\phi}(n) = \infty$ almost surely. We then define 
\begin{eqnarray*}
T_{n}^{\phi} := T(\tau^{\phi} (n)),
\end{eqnarray*} 

\noindent the CMJ-tree at time the total weight or ``size'' reaches $n$ (provided this ever happens). 
\end{definition}

Random trees $T_{n}^{\phi}$ defined in this way, for some CMJ-process and some weight $\phi$, are the focus of the present paper. For now on, 
we usually refer to $T_{n}^{\phi}$ as the CMJ-tree which size is given by $|T_{n}^{\phi}| := Z(\tau^{\phi} (n))$.
If $\phi \equiv 1$, $T_{n}^{\phi}$ is the family tree of a CMJ-process stopped when its number of vertices is greater than $n$. In particular, if the birth times have continuous distributions and there are no twins, then a.s.\ no two vertices are born simultaneously. Therefore, $|T_{n}^{\phi}| = n$.

Notice that $T_{n}^{\phi}$ could be an infinite random tree, or also the time $\tau^{\phi} (n)$ could be infinite. In order to avoid such possibilities, we only study cases where  $Z^{\phi}(t) < \infty$ for every finite $t \geq 0$, but $Z(\infty) = \infty$. In this direction, we define the Laplace transform of a function $f$ on $[0, \infty)$ by
 \begin{eqnarray*}
\hat{f}(\theta) = \theta \int_{0}^{\infty} e^{- \theta t} f(t) {\rm d} t, \hspace*{5mm} \theta >0,
\end{eqnarray*}
 
\noindent and the Laplace transform of a measure $\nu$ on $[0, \infty)$ by
\begin{eqnarray} \label{eq25}
\hat{\nu}(\theta) = \int_{0}^{\infty} e^{- \theta t} \nu({\rm d} t), \hspace*{5mm} - \infty < \theta < \infty.
\end{eqnarray}

\noindent Through this work and unless we specify otherwise, we make the following assumptions. 

\theoremstyle{Assumption} 
\begin{Assumption} 
We consider CMJ-processes that satisfy:
\begin{itemize}
 \item[({\bf A1})] $\mu(\{0\}) = \mathbb{E}[\Xi(\{ 0\})] < 1$. This excludes a trivial case with explosions at the start (in our examples, $\mu(\{0\}) = 0$).

 \item[({\bf A2})] $\mu$ is not concentrated on any lattice $h \mathbb{Z}$, $h > 0$.
 
\item[({\bf A3})] $N \geq 1$ a.s.\ (in this case, every individual has at least one child, so the process never dies out and $Z(\infty) = \infty$).
 
 \item[({\bf A4})] There exists $ \theta_{1} > 0$ such that $\hat{\mu}(\theta_{1}) \in (1, \infty)$. Notice that $\hat{\mu}(\cdot)$ is monotone decreasing on $(\theta_{1}, \infty)$, and that $\hat{\mu}(\theta) \rightarrow 0$, as $\theta \rightarrow \infty$, by the dominated convergence theorem. Therefore, there exists a real number $\alpha >0$ (the {\sl Malthusian parameter}) such that $\hat{\mu}(\alpha) =1$.
 
\item[({\bf A5})] For $\theta_{1}$ as in ({\bf A4}), we have that $Var(\hat{\Xi}(\theta_{1})) < \infty$. 

 \item[({\bf A6})] The random variable $\sup_{t \geq 0} \left( e^{-\theta_{2} t} \phi(t) \right)$ has finite expectation for some $0 < \theta_{2} < \alpha$.
 
\item[({\bf A7})] $Var(\phi(t))$ is bounded in finite intervals. Furthermore, there exists $0 < \theta_{3} \leq 2 \alpha$ such that $\lim_{t \rightarrow \infty} e^{-\theta_{3} t} Var(\phi(t)) = 0$.
\end{itemize}
\end{Assumption}

Observe that ({\bf A1})-({\bf A5}) are conditions on the general branching process, while ({\bf A6})-({\bf A7}) are conditions on the characteristic $\phi$. The following result shows that $T_{n}^{\phi}$ is well-defined. 

\begin{proposition} \label{Pro3}
Under the assumptions ({\bf A1})-({\bf A4}) and for any characteristic $\phi$ satisfying ({\bf A6}), we have that 
\begin{itemize}
\item[(i)] $\lim_{t \rightarrow \infty} Z^{\phi}(t) = \infty$ almost surely. Thus a.s.\ $\tau^{\phi}(n) < \infty$ for every $n \geq 0$ and $T_{n}^{\phi}$ is a well-defined finite random tree.

\item[(ii)] $\lim_{n \rightarrow \infty} |T_{n}^{\phi}|/n = 1/\mathbb{E}[\hat{\phi}(\alpha)] \in (0, \infty)$ a.s., and
$ \lim_{n \rightarrow \infty}  \tau_{n}^{\phi} / (\ln n) = 1/ \alpha$ almost surely.
\end{itemize}
\end{proposition}

\begin{proof}
See \cite[Theorem 5.12]{Holm2017}.
\end{proof}

We end this section by making a few remarks on our assumptions. 
Notice that ({\bf A4}) implies that $\mathbb{E}[N] > 1$ (this is know as the supercritical case). 
Notice also that ({\bf A4}) implies that $\mu(t) < \infty$ for every $0 \leq t < \infty$. However, $\mu(\infty) = \mathbb{E}[N]$ may be infinite. Furthermore, this condition also implies that $Z(t)$ and $\mathbb{E}[Z(t)]$ are finite for every $0 \leq t < \infty$; see for instance \cite[Theorem 6.3.3]{Ja}. Finally, we do not really need the assumption $N \geq 1$ in ({\bf A3}); it suffices that $\mathbb{E}[N] > 1$. In this case, the extinction probability $\mathbb{P}(Z(\infty) < \infty) < 1$, so there is a positive probability that the process is infinite, and Proposition \ref{Pro3} and the results below hold conditioned on the event $\{Z(\infty) = \infty\}$ (this is a standard setting in \cite{Ner1981}, \cite{JaNe} and \cite{Ner1984}).


\subsection{Main results} \label{Main}

We now consider Bernoulli bond-percolation with parameter $p_{n} \in (0, 1)$ on the CMJ-tree $T_{n}^{\phi}$ with given weight $\phi$ (recall Definition \ref{def1}). Following the idea of \cite{BeU2015}, we incorporate Bernoulli bond-percolation on the growth algorithm of the random tree process $(T(t), t \geq 0)$ in a dynamic way and stop at the time $\tau^{\phi}(n)$. This will lead us to interpret Bernoulli bond-percolation in terms of neutral mutations which are superposed to the structure of the CMJ-process and that appear at the birth events. More precisely, at each birth event, independently of all other individuals, the newborn is a clone of its parent with probability $p_{n}$ or a mutant with probability $1-p_{n}$.  The mutations are considered to be neutral, i.e., the behavior (reproduction laws and lifetimes) of the individuals is the same regardless of they are clones or mutants. A mutation event corresponds to the insertion of an edge in $T(t)$ that is immediately destroyed. This creates a new percolation cluster (with one vertex or individual) that grows following the same dynamic. We write $T^{(p_{n})}(t)$ for the resulting combinatorial structure at time $t \geq 0$. That is, $T^{(p_{n})}(t)$ has the same set of vertices as $T(t)$ and its set of intact edges is a subset of the edges of $T(t)$.  Thus, the connected clusters of $T^{(p_{n})}(t)$ are the subtrees of $T(t)$ formed by the subsets of vertices which can be connected by a path of intact edges. 

In this work, we are interested in the evolution of the percolation cluster that contains the root. In this direction, we write $T_{\varnothing}^{(p_{n})}(t)$ for the subtree of $T(t)$ at time $t \geq 0$ that contains the progenitor of the entire population at time $0$. It should be clear that the sub-population with the ancestral type is a CMJ-process whose generic birth process, denoted by $\Xi^{(p_{n})}$, has intensity measure given by
\begin{eqnarray} \label{eq3}
\mu^{(p_{n})}({\rm d} t) := p_{n} \mu ({\rm d }t),
\end{eqnarray}

\noindent where $\mu$ is the intensity measure of the birth process $\Xi$ of the original CMJ-process. This is a consequence of the thinning property of point processes. For a characteristic $\phi$, we denote by 
\begin{eqnarray*}
Z_{\varnothing}^{(p_{n}), \phi} = (Z_{\varnothing}^{(p_{n}), \phi}(t), t \geq 0), 
\end{eqnarray*}

\noindent the $\phi$-counted process associated with the (clonal) CMJ-process of the sub-population bearing the same type as the initial individual. In particular, if $\phi \equiv 1$, $Z_{\varnothing}^{(p_{n}), \phi} = Z_{\varnothing}^{(p_{n})} =  (Z_{\varnothing}^{(p_{n})}(t), t \geq 0)$ counts the number of vertices in the root cluster. Clearly, if $p_{n} \equiv 1$, we recover the original CMJ-process. Let 
\begin{eqnarray*}
T_{\varnothing}^{n,\phi} := T^{(p_{n})}_{\varnothing} (\tau^{\phi}(n))
\end{eqnarray*}

\noindent be the sub-tree which contains the original root of the CMJ-tree $T_{n}^{\phi}$ (associated to the weight $\phi$) after performing percolation of parameter $p_{n} \in (0,1)$. Recall that $\tau^{\phi} (n):= \inf \{t \geq 0: Z^{\phi}(t) \geq n \}$ is the first time that the total weight or ``size'' of the tree process $(T(t), t \geq 0)$ is at least $n$. Therefore, the size of the root percolation cluster is given by 
\begin{eqnarray*}
|T_{\varnothing}^{n,\phi}| : = Z_{\varnothing}^{(p_{n})}(\tau^{\phi}(n)).
\end{eqnarray*}

We turn now to the statement of our main result Theorem \ref{Teo2}. Recall that we consider the regimes weakly supercritical, supercritical and strongly supercritical of $p_{n} \in (0, 1)$, with $p_{n} \rightarrow 1$ as $n \rightarrow \infty$. First, we need to introduce some notation that we will use in the rest of the work. Notice that ({\bf A4}) implies that there exists $n^{\ast} \in \mathbb{N}$ such that for $n \geq n^{\ast}$ there is $\alpha_{p_{n}} >0$ (the {\sl Malthusian parameter} of $\mu^{(p_{n})}$) such that $\hat{\mu}^{(p_{n})}(\alpha_{p_{n}}) =1$. We write 
\begin{eqnarray} \label{eq27}
\bar{\mu}(\alpha) := \int_{0}^{\infty} t e^{-\alpha t} \mu({\rm d}t)
\end{eqnarray}

\noindent which is finite and strictly positive due to our assumptions; see Remarks \ref{rem4} and \ref{rem2} below.


\begin{theorem} \label{Teo2}
Let $\phi$ be any characteristic that does not depend on $p_{n}$. Under the assumptions ({\bf A1})-({\bf A7}), we have that
\begin{eqnarray*}
\lim_{n \rightarrow \infty} n^{-\frac{\alpha_{p_{n}}}{\alpha}}  |T_{\varnothing}^{n,\phi}|= \mathbb{E}[\hat{\phi}(\alpha)], \hspace*{5mm} \text{in probability}, 
\end{eqnarray*}

\noindent where $\alpha - \alpha_{p_{n}} \sim (1-p_{n}) \bar{\mu}(\alpha)^{-1}$ as $n \rightarrow \infty$. In particular, 
\begin{itemize}
\item[(i)] In the weakly supercritical regime, $\lim_{n \rightarrow \infty} n^{-1} |T_{\varnothing}^{n,\phi}| = 0$ in probability. 

\item[(ii)] In the supercritical regime,  $\lim_{n \rightarrow \infty} n^{-1} |T_{\varnothing}^{n,\phi}| = e^{-\frac{c}{\alpha \bar{\mu}(\alpha)}} \mathbb{E}[\hat{\phi}(\alpha)]$ in probability. 

\item[(iii)] In the strongly supercritical regime, $\lim_{n \rightarrow \infty} n^{-1} |T_{\varnothing}^{n,\phi}| = \mathbb{E}[\hat{\phi}(\alpha)]$ in probability. 
\end{itemize}
\end{theorem}

It is important to mention that the parameters $\alpha$ and $\alpha_{p_{n}}$ are difficult to compute explicitly. Nevertheless, in the supercritical and strongly supercritical regimes, we notice that $n^{\alpha_{p_{n}}\alpha^{-1}} \sim n^{ 1- (1-p_{n}) (\alpha\bar{\mu}(\alpha))^{-1}}$ as $n \rightarrow \infty$. We also notice that in the weakly supercritical regime, the size of the root cluster is $o(n)$ whereas in the other regimes it is of order $n$. 

On the other hand, in the supercritical regime (i.e., $1-p_{n} \sim c/ \ln n$ as $n \rightarrow \infty$, with $c >0$ fixed), we have that Theorem \ref{Teo2} (ii) agrees with \cite[Theorem 1]{Be2013}. However,
it must be pointed out that Theorem \ref{Teo2} (ii) cannot always be deduced from \cite[Theorem 1]{Be2013}. More precisely, notice that \cite[Theorem 1]{Be2013} considers size as the number of vertices in the tree structure whereas the size of a CMJ-tree depends on the underlying characteristic or weight (recall Definition \ref{def1}) that does not always coincides with the number of vertices. For instance, $m$-ary search trees, fragmentation trees or the so-called splitting trees where the notion of ``size'' is different; see Section \ref{applications} below for details. Then the arguments used in \cite{Be2013} do not always apply in our setting except in particular cases. For example, random recursive trees, preferential attachment trees or binary search trees that agree with a type of CMJ-tree when the correct birth process, lifespan and characteristic are chosen. Moreover, Betoin \cite{Be2013} did not treat the strongly and weakly supercritical regimes as we will do in this work. Therefore, Theorem \ref{Teo2} may be seen as a complementary (or extension) of the results in \cite{Be2013}. 

Inspired by Bertoin and Uribe Bravo \cite{BeU2015}, our approach relies on the connection between CMJ-processes with neutral mutations and Bernoulli bond-percolation. 
This leads us to investigate the asymptotic behaviour of a CMJ-process with neutral mutations up to a large random time, in certain regimes when the small mutation parameter is realted to size of the total population.  In \cite{BeU2015}, the authors connected Bernoulli bond-percolation in preferential attachment trees with a Markovian system of branching processes with neutral mutations. This is clearly not the case here since it is well-known that CMJ-processes are not always Markovian. Thus, we have to use different tools, although some guidelines are similar to \cite{BeU2015}. We stress that similar connections with systems of (Markovian) branching processes have been used before to study percolation on random recursive trees \cite{Er2016}, \cite{Eri2017} and $m$-ary random increasing trees \cite{Ber2015}. 

This work leaves some open natural questions that we plan to investigate in the future. One can consider estimating the sizes of the largest clusters which do not contain the root. In this work, we restrict ourselves to the root cluster because the absence of Markov property makes the analysis much harder. This is not the case in \cite{BeU2015} and  \cite{Er2016} where the connection with a Markovian branching system with neutral mutations is used to answer this question for random recursive trees and preferential attachment trees. We refer also to Bertoin \cite{Be2014} where this question has been answered for random recursive trees by using a different approach. The second direction of future work would be to analyze the fluctuations of the giant component that we expect to be non-Gaussian as for random recursive trees \cite{Be22014}, preferential attachment trees and $m$-ary random increasing trees \cite{Ber2015}. Finally, it would be interested to estimate the size of the largest percolation clusters in the sub-critical regime, i.e., $1 - p_{n} \gg c/ \ln n$ as $n \rightarrow \infty$ and $c > 0$ is fixed; see for instance \cite{Eri2017}, where the case of the random recursive tree has been studied. 

The rest of this paper is organized as follows: In Section 
\ref{ProofM}, we prove our main result. Section \ref{applications} is devoted to the application of Theorem \ref{Teo2} to important families of random trees that can be constructed via CMJ-processes. Finally, the key results used in the proof of Theorem \ref{Teo2} are proven in Section \ref{prel}. More precisely, we investigate the asymptotic behaviour of CMJ-processes with mutations and deduce some crucial properties that may be of independent interest. 

\section{Proof of Theorem \ref{Teo2}} \label{ProofM}

In this section, we prove our main result Theorem \ref{Teo2}. Recall that the size of the root cluster is related to the clonal CMJ-process with generic birth process $\Xi^{(p_{n})}$ whose intensity measure $\mu^{(p_{n})}$ is given in (\ref{eq3}). The starting point is to investigate the asymptotic behavior of the $\phi$-counted clonal process $Z_{\varnothing}^{(p_{n}), \phi} = (Z_{\varnothing}^{(p_{n}), \phi}(t), t \geq 0)$ as $p_{n} \rightarrow 1$ and $t \rightarrow \infty$. The approach relies crucially on the use of a remarkable martingale that can be found in the work of Nerman \cite{Ner1981}. More precisely, we improve the results in \cite{Ner1981} and \cite{JaNe} on the convergence of Nerman's martingale in order to hold uniformly in the percolation parameter (see Lemmas \ref{lemma1} and \ref{lemma2} below). Then these results and further remarks are put together to conclude with the proof of Theorem \ref{Teo2}. 

Recall that $({\bf A4})$ implies that there exists $n^{\ast} \in \mathbb{N}$ such that for $n \geq n^{\ast}$ there is $\alpha_{p_{n}} > 0$ (the {\sl Malthusian parameter} of $\mu^{(p_{n})}$) such that $\hat{\mu}^{(p_{n})}(\alpha_{p_{n}}) =1$. This implies that $\mu^{(p_{n})}(\infty) = p_{n} \mathbb{E}[N] > 1$ (i.e. the clonal process is supercritical). Since $\alpha_{p_{n}} \rightarrow \alpha$ as $p_{n} \rightarrow 1$, we choose $n^{\ast}$ large enough such that $0 < \theta_{1} < \inf_{n \geq n^{\ast}} \alpha_{p_{n}}$ where $\theta_{1}$ satisfies $({\bf A4})$-$({\bf A5})$. Moreover, consider $n^{\ast}$ even larger such that 
$0 < \theta_{2} < \inf_{n \geq n^{\ast}} \alpha_{p_{n}}$ and $0 < \theta_{3} < 2\inf_{n \geq n^{\ast}} \alpha_{p_{n}}$, where $\theta_{2}$ and $\theta_{3}$ satisfy $({\bf A6})$ and $({\bf A7})$, respectively.

\begin{remark} \label{rem4}
Notice for future reference that $\mu_{\alpha}^{(p_{n})}(dt) := e^{-t\alpha_{p_{n}}} \mu^{(p_{n})}(dt)$, for $n \geq n^{\ast}$, is a probability measure concentrated on $(0, \infty)$. Moreover, condition $({\bf A4})$ implies that
\begin{eqnarray} \label{eq35}
\bar{\mu}^{(p_{n})}(\alpha_{p_{n}}) : = \int_{0}^{\infty} t \mu_{\alpha}^{(p_{n})}({\rm d} t) < \infty.
\end{eqnarray}
\end{remark}

We write $W^{(p_{n}), \phi}_{\varnothing} = (W^{(p_{n}), \phi}_{\varnothing}(t), t \geq 0)$ for the process given by
\begin{eqnarray*}
 W^{(p_{n}), \phi}_{\varnothing}(t) := e^{-t\alpha_{p_{n}} } Z^{(p_{n}), \phi}_{\varnothing}(t), \hspace*{6mm} t \geq 0.
\end{eqnarray*}


\noindent For $p_{n} \equiv 1$, we sometimes remove the superscript $(p_{n})$ and the subscript $\varnothing$ from the previous notations. That is, we write $W^{\phi} = (W^{\phi}(t), t \geq 0)$ for the process given by  $W^{\phi}(t) := e^{-t\alpha } Z^{\phi}(t)$.


For $n \geq n^{\ast}$, consider the characteristic 
\begin{eqnarray} \label{eq10}
\psi^{(p_{n})}(t) = \mathds{1}_{\{ t \geq 0 \}} e^{t\alpha_{p_{n}}} \int_{t}^{\infty} e^{-s\alpha_{p_{n}} } \Xi^{(p_{n})}({\rm d} s), \hspace*{5mm} t \geq 0.
\end{eqnarray}

\noindent The next lemma shows that $\psi^{(p_{n})}$ satisfies the conditions $({\bf A6})$-$({\bf A7})$. 

\begin{lemma} \label{lemma4}
Assume that conditions $({\bf A1})$-$({\bf A5})$ are fulfilled. Then,
\begin{itemize}
\item[(i)]  $\sup_{n \geq n^{\ast}}  \sup_{t \geq 0} e^{-\theta_{1} t} \psi^{(p_{n})}(t)$ has finite expectation, for $\theta_{1}$ as in $({\bf A4})$.

\item[(ii)] $\sup_{n \geq n^{\ast}}  Var(\psi^{(p_{n})}(t))$ is bounded in finite intervals. Furthermore, there exists $0 < \theta \leq 2 \inf_{n \geq n^{\ast}}\alpha_{p_{n}}$ such that $\lim_{t \rightarrow \infty} \sup_{n \geq n^{\ast}}e^{-\theta t}  Var(\psi^{(p_{n})}(t)) = 0$.
\end{itemize}
\end{lemma}

\begin{proof}
For $t \geq 0$ and $0 < \theta < \alpha_{p_{n}}$, we notice that $\psi^{(p_{n})}(t) \leq e^{t\theta} \int_{t}^{\infty} e^{-s\theta} \Xi({\rm d} s) \leq e^{t\theta} \hat{\Xi}(\theta)$. This inequality and conditions $({\bf A4})$-$({\bf A5})$ imply our claim. 
\end{proof}

We henceforth, and for sake of simplicity, omit the superscript $(p_{n})$ from $\psi^{(p_{n})}$ and only write $\psi$ for the characteristic defined in (\ref{eq10}). 

It is well-known that the process $W^{(p_{n}), \psi}_{\varnothing}$ is a nonnegative square-integrable martingale whose terminal value will be denoted by $W^{(p_{n}), \psi}_{\varnothing}(\infty)$. Furthermore, $W^{(p_{n}), \psi}_{\varnothing}(\infty) \geq  0$ almost surely (cf. \cite[Proposition 2.4]{Ner1981} for the proof of the martingale property and \cite[Theorem 4.1 and Corollary 4.2]{JaNe} for the convergence result). In particular, if $p_{n}\equiv1$, \cite[Corollary 4.2]{JaNe} and condition $({\bf A3})$ imply that $W^{(p_{n}), \psi}(\infty) = W^{\psi}(\infty) >0$ almost surely. 

An important result established by Nerman \cite{Ner1981} and Jagers and Nerman \cite[Theorem 4.3]{JaNe} (see also Jagers \cite[Section 6.10]{Ja}) implies that for $n \geq n^{\ast}$, 
\begin{eqnarray} \label{eq36}
\lim_{t \rightarrow \infty} W^{(p_{n}), \phi}_{\varnothing}(t) = m_{\infty}^{(p_{n}), \phi} W^{(p_{n}), \psi}_{\varnothing}(\infty), 
\end{eqnarray}

\noindent almost surely and in $L_{2}(\mathbb{P})$, where
\begin{eqnarray} \label{eq39}
m_{\infty}^{(p_{n}), \phi} := \frac{\mathbb{E}[\hat{\phi}(\alpha_{p_{n}})]}{\alpha_{p} \bar{\mu}^{(p_{n})}(\alpha_{p_{n}})} < \infty,
\end{eqnarray}

\noindent with $\bar{\mu}^{(p_{n})}(\alpha_{p_{n}})$ defined in (\ref{eq35}).

\begin{remark} \label{rem2}
\noindent Notice that the previous convergence implies that $\bar{\mu}^{(p_{n})}(\alpha_{p_{n}}) >0$. Furthermore, $m_{\infty}^{(p_{n}), \phi} > 0$ (or equivalently, $\mathbb{E}[\hat{\phi}(\alpha_{p_{n}})] > 0$) whenever  $\phi \not \equiv 0$ almost surely.
\end{remark}

\begin{remark} \label{rem3}
In particular, a simple computation shows that $
m_{\infty}^{(p_{n}),\psi} =1$, where $\psi$ is defined in (\ref{eq10}); see for instance \cite[Theorem 4.1]{JaNe}.
\end{remark}

The next two results are the key ingredients in the proof of Theorem \ref{Teo2}. The first lemma shows that the $L_{2}(\mathbb{P})$ convergence of the square-integrable martingale $W^{(p_{n}), \psi}_{\varnothing}$ holds uniformly for $n \geq n^{\ast}$. Furthermore, it shows that the $L_{2}(\mathbb{P})$  convergence in (\ref{eq36}) also holds uniformly for $n \geq n^{\ast}$. The second lemma establishes an even stronger convergence result by showing that the convergence remains true as $t \rightarrow \infty$ and $p_{n} \rightarrow 1$. 

\begin{lemma} \label{lemma1}
Assume that conditions $({\bf A1})$-$({\bf A5})$ are fulfilled. We have that
\begin{eqnarray*}
\lim_{t \rightarrow \infty} \sup_{n \geq n^{\ast}} \mathbb{E} \left[ \sup_{s\geq t} \left| W^{(p_{n}), \psi}_{\varnothing}(s) - W^{(p_{n}), \psi}_{\varnothing}(\infty)  \right|^{2} \right] = 0,
\end{eqnarray*}

\noindent where $\psi$ is defined in (\ref{eq10}). Furthermore, let $\phi$ be a characteristic that does not depend on $p_{n}$ and that satisfies $({\bf A6})$-$({\bf A7})$, then we have that
\begin{eqnarray*}
\lim_{t \rightarrow \infty} \sup_{n \geq n^{\ast}} \mathbb{E} \left[  \left| W^{(p_{n}), \phi}_{\varnothing}(t) - m_{\infty}^{(p_{n}), \phi}W^{(p_{n}), \psi}_{\varnothing}(\infty) \right|^{2} \right] = 0.
\end{eqnarray*}
\end{lemma}

\begin{lemma} \label{lemma2}
Assume that conditions $({\bf A1})$-$({\bf A5})$ are fulfilled. For a characteristic $\phi$ that does not depend on $p_{n}$ and that satisfies $({\bf A6})$-$({\bf A7})$, we have that
\begin{eqnarray*}
\lim_{n \rightarrow \infty, t \rightarrow \infty} \mathbb{E} \left[ \left| W_{\varnothing}^{(p_{n}), \phi}(t) -  m_{\infty}^{\phi} W^{\psi}(\infty) \right |^{2}  \right] = 0,
\end{eqnarray*}

\noindent where $\psi$ is defined in (\ref{eq10}) with $p_{n} \equiv 1$ (the limit must be understood as a double limit). 
\end{lemma}

The proofs of these lemmas are rather technical and it is convenient to postpone their proofs until the Appendix \ref{prel}. We then finish the proof of Theorem \ref{Teo2}. But first we need the next result. 

\begin{lemma} \label{lemma3}
Assume that conditions ({\bf A1})-({\bf A4}) are fulfilled. We have that $\alpha - \alpha_{p_{n}} \sim  (1-p_{n}) \bar{\mu}(\alpha)^{-1}$, as $n \rightarrow \infty$, where $\bar{\mu}(\alpha)$ is defined in (\ref{eq27}). 
\end{lemma}

\begin{proof}
Recall the definition of $\hat{\mu}(\cdot)$ in (\ref{eq25}). Recall also that ({\bf A4}) implies that $\hat{\mu}(\cdot)$ is a continuous monotone decreasing function on $(\theta_{1}, \infty)$, where $\theta_{1}$ is defined in  $({\bf A4})$. Furthermore, ({\bf A4}) and the dominated convergence theorem show that $\hat{\mu}(\cdot)$ is differentiable with continuous derivative given by 
\begin{eqnarray} \label{eq38}
\hat{\mu}^{\prime}(\theta) := \frac{{\rm d} }{{\rm d} s} \hat{\mu}(s) \bigg |_{s = \theta} =  - \int_{0}^{\infty} t e^{-\theta t} \mu({\rm d} t), \hspace*{5mm} \text{for} \hspace*{3mm} \theta \in (\theta_{1}, \infty).
\end{eqnarray}

\noindent Since $p_{n} \in (0,1)$, we have that $\alpha_{p_{n}} < \alpha$ for $n \geq n^{\ast}$. Then,  for $n \geq n^{\ast}$, the mean value theorem implies that there exists $\varepsilon_{n} \in (\alpha_{p_{n}}, \alpha)$ such that
\begin{eqnarray*}
\hat{\mu}^{\prime}(\varepsilon_{n}) = \frac{\hat{\mu}(\alpha) - \hat{\mu}(\alpha_{p_{n}})}{\alpha - \alpha_{p_{n}}}.
\end{eqnarray*}

\noindent Recall that ({\bf A4}) implies that $\hat{\mu}(\alpha) = 1$ and $\hat{\mu}(\alpha_{p_{n}}) = 1/p_{n}$. Moreover, we have that $0 < \bar{\mu}(\alpha) < - \hat{\mu}^{\prime}(\varepsilon_{n})  < \infty$ where $\bar{\mu}(\alpha)$ is defined in (\ref{eq27}). Hence
\begin{eqnarray*}
\alpha - \alpha_{p_{n}} = - \hat{\mu}^{\prime}(\varepsilon_{n})^{-1} \frac{1-p_{n}}{p_{n}}.
\end{eqnarray*}

\noindent Finally, one concludes from the continuity of $\hat{\mu}^{\prime}(\cdot)$ since $\alpha_{p_{n}} \rightarrow \alpha$, as $ n \rightarrow \infty$, and $ - \hat{\mu}^{\prime}(\alpha) = \bar{\mu}(\alpha)$.
\end{proof}

\begin{proof}[Proof of Theorem \ref{Teo2}] 
We deduce from Lemmas \ref{lemma1} and \ref{lemma2} that
\begin{eqnarray*} 
\lim_{n \rightarrow \infty} e^{-\alpha \tau_{n}^{\phi}} Z(\tau_{n}^{\phi}) = \lim_{n \rightarrow \infty} e^{-\alpha_{p_{n}} \tau_{n}^{\phi}} Z_{\varnothing}^{(p_{n})}(\tau_{n}^{\phi}) = m_{\infty} W^{\psi}(\infty),
\end{eqnarray*}

\noindent in probability, where $\psi$ is defined in (\ref{eq10}) with $p_{n} \equiv1$ and $m_{\infty} = (\alpha \bar{\mu}(\alpha))^{-1} \in (0, \infty)$; see Remark \ref{rem2}. Proposition \ref{Pro3} implies that $\lim_{n \rightarrow \infty} n^{-1} Z(\tau_{n}^{\phi}) = \mathbb{E}[\hat{\phi}(\alpha)]$ a.s., then
\begin{eqnarray*}
\lim_{n \rightarrow \infty} n^{\frac{\alpha_{p_{n}}}{\alpha}} e^{-\alpha_{p_{n}} \tau_{n}^{\phi}} =  m_{\infty} \mathbb{E}^{-1}[\hat{\phi}(\alpha)] W^{\psi}(\infty).
\end{eqnarray*}

\noindent Since $W^{\psi}(\infty) > 0$ almost surely (by $({\bf A3})$ together with \cite[Corollary 4.2]{JaNe}), we obtain that
\begin{eqnarray*}
\lim_{n \rightarrow \infty} n^{-\frac{\alpha_{p_{n}}}{\alpha}} Z_{\varnothing}^{(p_{n})}(\tau_{n}^{\phi})= \mathbb{E}[\hat{\phi}(\alpha)],
\end{eqnarray*}

\noindent in probability. Therefore, our claim follows by Lemma \ref{lemma3}.
\end{proof}

\section{Applications} \label{applications}

In this section, we apply Theorem \ref{Teo2} to deduce known and new results on the existence of a giant percolation cluster in the supercritical regime for several families of trees. We begin in Subsection \ref{general} with some examples where the result has been established in \cite{Be2013}. In Subsections \ref{Ex2}, \ref{Median}, \ref{fragmentation} and \ref{Homo}, we present several new examples.

\subsection{General preferential attachment trees} \label{general}

We introduce the procedure studied by Rudas, T\'oth and Valk\'o \cite{Ru2007} and Rudas and T\'oth \cite{Ru2009} to grow a so-called general preferential attachment tree. Fix a sequence of nonnegative weights ${\bf w} = (w_{k})_{k=0}^{\infty}$ with $w_{0} >0$. Start the construction from an unique tree with a single vertex and build a random tree $T_{n}^{({\bf w})}$ with $n$ vertices recursively as follows. Suppose that $T_{n}^{({\bf w})}$ has been constructed for $n \geq 1$, and for every vertex $v \in T_{n}^{({\bf w})}$ denote by $d_{n}^{+}(v)$ its outdegree. Given $T_{n}^{({\bf w})}$, the tree $T_{n+1}^{({\bf w})}$ is derived from $T_{n}^{({\bf w})}$ by incorporating a new vertex $u$ and creating an edge between $u$ and a vertex $v_{n} \in T_{n}^{({\bf w})}$ chosen at random according to the law
\begin{eqnarray*}
\mathbb{P}(v_{n} = v | T_{n}^{({\bf w})}) = w_{d_{n}^{+}(v)} \bigg( \sum_{v^{\prime}} w_{d_{n}^{+}(v^{\prime})} \mathds{1}_{\{v^{\prime} \in T_{n}^{({\bf w})} \}}\bigg)^{-1}, \hspace*{5mm} \text{for} \hspace*{3mm} v \in T_{n}^{({\bf w})}.
\end{eqnarray*}

\noindent It is important to point out that different choices of the sequence $\mathbf{w}$ yield to well-known families of trees. Some of these families are summarized in Table \ref{table1}; see \cite{Szy1987}, \cite{Aldous1991}, \cite{Mah1994}, and \cite{Boris1994} for background.  

\begin{table}[h!]
\centering
\caption{Examples of general preferential attachment trees.}
\label{table1}
\begin{tabular}{|c|c|}
\hline
 & ${\bf w} = (w_{k})_{k=0}^{\infty}$  \\ \hline
random recursive tree & $w_{k} = 1$ for all $k \geq 0$ \\ \hline
binary search tree  & $w_{0} = 2$, $w_{1} =1$ and $w_{k} = 0$ for $k \geq 2$ \\ \hline
$m$-ary increasing tree ($m \geq 2$) & $w_{k} = m -k$, for $k =0,1,\dots, m-1$, and $w_{k} = 0$ for $k \geq m$  \\ \hline
linear preferential attachment &  $w_{k} = \beta k + \rho$, where $\beta \in \{-1, 0, 1 \}$ and $\rho \in \mathbb{R}_{+} \setminus \{0 \}$ \\ \hline
binary pyramid &  $w_{0} = w_{1} = 1$ and $w_{k} = 0$ for $k \geq 2$. \\ \hline
\end{tabular}
\end{table}

Preferential attachment trees can be constructed via CMJ-processes as in Definition \ref{def1}. More precisely, consider the characteristic (or weight) $\phi \equiv 1$ and a CMJ-process with birth times $\xi_{i} = \sum_{k=1}^{i} X_{k}$ for $i \in \mathbb{N} \cup \{0\}$ (with the convention $\xi_{0} = \sum_{k=1}^{0} X_{k} =0$), where $X_{i} = \xi_{i} - \xi_{i-1}$, for $i \in \mathbb{N}$, are independent random variables and distributed according as an exponential random variable of parameter $w_{i-1}$.
In other words, we have that the process $(\Xi([0,t]), t\geq 0)$ is a pure birth process starting at $0$ with birth rate $w_{k}$ when the state is $k \in \mathbb{N} \cup \{0\}$. In this example, the lifetime of the individuals $\lambda \equiv \infty$. In the sequel, we assume that the pure birth process $(\Xi([0,t]), t\geq 0)$ is non-explosive, i.e.,
\begin{eqnarray} \label{eq20}
\sum_{k=0}^{\infty} \frac{1}{w_{k}} = \infty;
\end{eqnarray}

\noindent see \cite{Athreya2007} for details. This implies that each individual in the CMJ-process has a.s.\ a finite number of children in each finite interval. Furthermore, notice that $\hat{\Xi}(\theta) = \sum_{k=1}^{\infty} e^{-\theta \xi_{k}}$ for $\theta >0$. Then,
\begin{eqnarray*}
\hat{\mu}(\theta) = \sum_{k=1}^{\infty} \mathbb{E}[e^{-\theta \xi_{k}}] = \sum_{k=1}^{\infty} \prod_{i=1}^{k} \mathbb{E}[e^{-\theta X_{i}}] = \sum_{k=1}^{\infty} \prod_{i=0}^{k-1} \frac{1}{1+\theta/w_{i}}.
\end{eqnarray*}

\noindent Assume that 
\begin{equation} \label{cond1}
\text{there exists} \hspace*{4mm} \varepsilon_{1} >0 \hspace*{4mm} \text{such that} \hspace*{4mm} 1 < \hat{\mu}(\varepsilon_{1}) < \infty. \tag{{\bf E1}}
\end{equation}

\noindent This implies that $w_{1} > 0$ and that the condition of non-explosion (\ref{eq20}) is fulfilled. At the same time, it should be plain that condition $({\bf A4})$ is verified. Thus, the {\sl Malthusian parameter} exists, i.e., there is $\alpha > \varepsilon_{1}$ such that $\hat{\mu}(\alpha) =1$. Finally, we further assume that
\begin{equation} \label{cond2}
 Var ( \hat{\Xi}(\varepsilon_{1})) < \infty \tag{{\bf E2}}.
\end{equation}

\noindent Hence the conditions $({\bf A1})$–$({\bf A7})$ are satisfied and Theorem \ref{Teo2} implies the following result.

\begin{corollary}  \label{cor1}
In the supercritial regime, under $({\bf E1})$-$({\bf E2})$, we have that $\lim_{n \rightarrow \infty} n^{-1} |T_{n}^{({\bf w})}| = e^{-\frac{c}{\alpha \bar{\mu}(\alpha)}}$ in probability.
\end{corollary}

Finally, Corollary \ref{cor1} allows us to recover some of the cases studied in \cite{Be2013}; see \cite[Section 6]{Holm2017} for details of the calculations.

\begin{table}[h!]
\centering
\label{table2}
\begin{tabular}{|c|c|c|c|}
\hline
 & $\hat{\mu}(\theta)$  & $\alpha$ & $\bar{\mu}(\alpha)$  \\ \hline
random recursive tree  & $\frac{1}{\theta}$, $\theta >0$ & $1$ & $1$   \\ \hline
binary search tree &  $\frac{2}{\theta+1}$, $\theta >-1$ & $1$ & $\frac{1}{2}$   \\ \hline
$m$-ary increasing tree ($m \geq 2$) & $\frac{m}{\theta+1}$, $\theta >-1$ &  $m-1$ & $\frac{1}{m}$ \\ \hline
linear preferential attachment & $\frac{\rho}{\theta - \beta}$, $\theta > \beta$ & $\beta + \rho$ & $\frac{1}{\rho}$ \\ \hline
binary pyramid & $\frac{1}{1+\theta} + \frac{1}{(1+\theta)^{2}}$, $\theta > -1$ & $\frac{\sqrt{5}-1}{2}$ & $\frac{4 \sqrt{5}}{(1+\sqrt{5})^{2}}$ \\ \hline
\end{tabular}
\end{table}

\subsection{The $m$-ary search tree} \label{Ex2}

The $m$-ary search trees,  where $m \geq 2$ is a fixed integer,  were first introduced in \cite{Mun1971}. In particular, $m=2$ corresponds to the binary search tree described in Section \ref{general}, Table \ref{table1}. An $m$-ary search tree is an $m$-ary tree constructed recursively from a sequence of keys (real numbers), where each vertex stores up to $m-1$ keys. More precisely, one starts from a tree containing just an empty vertex (the root). Assume that the keys are i.i.d.\ random variables with a continuous distribution on $\mathbb{R}$. Then add keys one by one until the $(m-1)$-th key is placed in the root (i.e., the root becomes full) and add $m$ new empty vertices as children of the root. Furthermore, the $m-1$ keys in the root divide the set of real numbers into $m$ intervals $I_{1}, \dots, I_{m}$ that one associates with each of the $m$ children of the root. Then each further key is passed to one of the children of the root depending on which interval it belongs, i.e., a key in $I_{i}$ is stored in the $i$-th child. Finally, one continues by iterating this procedure in an obvious way everytime a vertex becomes full. 

This construction yields the extended $m$-ary search tree. In this setting, the vertices containing at least one key are called internal and the empty vertices are called external. In this work, we decide to eliminate the external vertices and consider the tree consisting of the internal nodes only. This is the $m$-ary search tree (in any case, our results also apply to extended $m$-ary search tree). We also consider $m$-ary search trees with a fixed number of keys, say $n \in \mathbb{N}$. In other words, one stops the previous procedure at time when the $n$-th key is added and denotes by $T_{n}^{(m)}$ the resulting (random) $m$-ary search tree with $n$ keys. Notice that the number of vertices of $T_{n}^{(m)}$  is actually random. 

Following \cite[Section 7.2]{Holm2017}, one can construct $T_{n}^{(m)}$ as the family tree of a CMJ-process. Consider a continuous time version of the construction procedure of an $m$-ary search tree and start with one vertex (the root) with a single key. The root acquires more keys after successive independent waiting times $Y_{2}, \dots, Y_{m-1}$, where $Y_{i}$ is an exponential random variable of parameter $i \in \{2, \dots, m\}$. At the arrival of the $(m-1)$-th key, at time $\sum_{i=2}^{m-1} Y_{i}$ (with the convention that the sum is equal to $0$ when $m=2$), the root gets $m$ children with one key each of them, marked by $1, \dots, m$, with the child $k$ born after a further waiting time $X_{k}$, i.e.\ at time $\sum_{i=2}^{m-1} Y_{i} + X_{k}$, where $X_{1}, \dots, X_{m}$ are independent and exponentially distributed random variables of parameter $1$. Finally, one continues growing the tree in an obvious way. Clearly, the CMJ-process associated with an $m$-ary search tree possesses birth times $\xi_{k} = \sum_{i=2}^{m-1} Y_{i} + X_{k}$, for $k =1, \dots, m$, (note that $N=m$ is non-random) and life time of each individual $\lambda \equiv \infty$.  By considering the characteristic
\begin{eqnarray*}
\phi_{m}(t) = k, \hspace*{5mm} \text{for} \hspace*{5mm} \sum_{i=2}^{k}Y_{i} \leq t  < \sum_{i=2}^{k+1} Y_{i}, \hspace*{5mm} \text{and} \hspace*{3mm} k =1, 2, \dots, m-1, \hspace*{2mm} t\geq 0,
\end{eqnarray*}

\noindent one sees that $T^{\phi_{m}}_{n}$, in Definition \ref{def1}, is a random $m$-ary search tree with $n$ keys. Notice that $\hat{\Xi}(\theta) = \sum_{k=1}^{m} \exp(-\theta ( \sum_{i=2}^{m-1} Y_{i} + X_{k} ))$. Hence a simple computation implies that 
\begin{eqnarray} \label{eq17}
\hat{\mu}(\theta) = \sum_{k=1}^{m} \mathbb{E} \left[\exp \left(-\theta \left( \sum_{i=2}^{m-1} Y_{i} + X_{k} \right) \right) \right] =  m ! \prod_{i=1}^{m-1} (i+\theta)^{-1}, \hspace*{5mm} \theta > -1.  
\end{eqnarray}
 
\noindent In particular, we see that the {\sl Malthusian parameter} is $\alpha =1$. Furthermore,
\begin{eqnarray*}
\mathbb{E} \left[\hat{\Xi}(\theta)^{2} \right] = \hat{\mu}(2\theta) \left(  1 + \frac{(m-1)(m-2)}{m} \frac{1+2\theta}{(1+\theta)^{2}}\right), \hspace*{5mm} \theta > -1/2,
\end{eqnarray*}

\noindent which implies that $Var(\hat{\Xi}(\theta)) < \infty$ for $\theta > -1/2$. Thus, it should be clear that conditions ({\bf A1})–({\bf A7}) are satisfied. Consequently, Theorem \ref{Teo2} implies the following result. 

\begin{corollary}
In the supercritical regime, we have that $\lim_{n \rightarrow \infty} n^{-1} |T^{(m)}_{n}| = 2(H_{m}-1) e^{-\frac{c}{H_{m}-1}}$ in probability, where $H_{m} = \sum_{i=1}^{m} i^{-1}$.
\end{corollary}

\begin{proof}
The result follows from Theorem \ref{Teo2} by computing $\bar{\mu}(1)$ and $\mathbb{E}[\hat{\phi}_{m}(1)]$. Notice that (\ref{eq38}) and (\ref{eq17}) show that $\bar{\mu}(1) = - \hat{\mu}^{\prime}(1) = H_{m}-1$. Notice also that 
\begin{eqnarray*}
\hat{\phi}_{m}(1) = \int_{0}^{\infty} e^{-t} \phi_{m}(t) {\rm d} t = 1 + \sum_{k=2}^{m-1} e^{- \sum_{i=2}^{k} Y_{i}}.
\end{eqnarray*}

\noindent Therefore, a direct computation shows that
\begin{eqnarray*}
\mathbb{E}[\hat{\phi}_{m}(1)] = 1 +  \sum_{k=2}^{m-1} \mathbb{E}[e^{- \sum_{i=2}^{k} Y_{i}}]  = 1 +  2 \sum_{k=2}^{m-1} \frac{1}{k+1}  = 2(H_{m}-1),
\end{eqnarray*}

\noindent which concludes the proof.
\end{proof}
 
\subsection{Median-of-$(2 \ell+ 1)$ binary search tree} \label{Median}

The random median-of-$(2\ell+1)$ binary search tree, for $\ell \in \mathbb{N}$, (see e.g.\ \cite{Dev1993}) is a modification of the binary search tree (or $2$-ary search tree), where each internal vertex contains exactly one key, but each external one can contain up to $2\ell$ keys (recall that keys are real numbers). 

This tree is constructed recursively from an initial tree with a single external vertex without any keys. Then one adds keys one by one until the $(2\ell+1)$-th key is placed at this first external vertex  (or to another external one later in the process). The vertex becomes an internal one with two new external vertices as its children, say $v_{L}$ and $v_{R}$. Immediately, the median of the $2\ell+1$ keys at the vertex is computed and put at the external vertex, while the $\ell$ keys that are smaller than the median are put in the left child $v_{L}$ and the $\ell$ keys that are larger than the median are put in the right child $v_{R}$.  One continues by adding new keys to the root and send them to the left or to the right whenever they are smaller or larger than the median of the first $2\ell+1$ keys. Finally, one iterates this procedure in an obvious way in order to grow the tree until $n \in \mathbb{N}$ keys have been added. We denote by $T_{n}^{(\ell)}$ the random median-of-$(2\ell + 1)$ binary search tree with $n$ keys. 

Following \cite[Section 8]{Holm2017}, one can construct a median-of-$(2\ell + 1)$ binary search tree via a CMJ-process. Start with one vertex (the root) with $\ell$ keys (notice that this is not a problem because the first $\ell$ keys always go there). Then each external vertex will contain between $\ell$ and $2\ell$ keys, throughout the process. On the other hand, a vertex acquires $\ell+1$ additional keys after successive independent waiting times $Y_{1}, \dots, Y_{\ell+1}$, where $Y_{i}$ has exponential distribution of parameter $\ell+i$, for $i=1, \dots, \ell+1$. At the time the $(\ell+1)$-th key arrives, the vertex immediately gets $2$ children where each of them contains $\ell$ keys. Therefore, it should be clear that the CMJ-process related to the median-of-$(2\ell+ 1)$ binary search tree has birth times $\xi_{1}=\xi_{2} = \sum_{i=1}^{\ell+1} Y_{i}$ (in distribution) and where each individual has life time $\lambda \equiv \infty$. In this case, $N=2$. Define the characteristic 
\begin{eqnarray*}
\phi_{\ell}(t) = \left\{ \begin{array}{lcl}
             \ell+k, & \mbox{  } & \sum_{i=1}^{k} Y_{i} \leq t < \sum_{i=1}^{k+1} Y_{i}, \, \, \, \, 0 \leq k \leq \ell, \\
              1  & \mbox{  } & \sum_{i=1}^{\ell+1} Y_{i} \leq t, \\
              \end{array}
    \right.
\end{eqnarray*}

\noindent for $t \geq 0$ (with the convention $\sum_{i=1}^{0}Y_{i} = 0$). Therefore, the tree $T_{n}^{\phi_{\ell}}$, in Definition \ref{def1} is a median-of-$(2\ell + 1)$ binary search tree with $n$ keys. In this example, notice that $\hat{\Xi}(\theta) = 2\exp(-\theta \sum_{i=1}^{\ell+1} Y_{i})$. Hence,
\begin{eqnarray} \label{eq18}
\hat{\mu}(\theta) = \mathbb{E} \left[ 2\exp \left(-\theta \sum_{i=1}^{\ell+1} Y_{i} \right)\right] =2 \prod_{i=1}^{\ell+1} \frac{\ell+i}{\ell+i+\theta}, \hspace*{5mm} \theta > - (\ell +1).
\end{eqnarray}

\noindent We deduce that the {\sl Malthusian parameter} is $\alpha =1$. Furthermore, it is not difficult to see that 
\begin{eqnarray*}
Var (\hat{\Xi}(\theta)) = 2 \hat{\mu}(2\theta) - \hat{\mu}(\theta)^{2} < \infty, \hspace*{5mm} \theta > - (\ell +1)/2.
\end{eqnarray*}

\noindent Thus, it should be plain that all the conditions $({\bf A1})$–$({\bf A7})$ are satisfied. Consequently, Theorem \ref{Teo2} implies the following result.  

\begin{corollary}
In the supercritical regime, we have that $\lim_{n \rightarrow \infty} n^{-1} |T_{n}^{(\ell)}| = (\ell +1)(H_{2 \ell +2}-H_{\ell +1}) e^{-c(H_{2 \ell +2}-H_{\ell +1})^{-1}}$ in probability. 
\end{corollary}

\begin{proof}
By Theorem \ref{Teo2}, we only need to compute $\bar{\mu}(1)$ and $\mathbb{E}[\hat{\phi}_{\ell}(1)]$. Notice that (\ref{eq38}) and (\ref{eq18}) show that $\bar{\mu}(1) = - \hat{\mu}^{\prime}(1) = H_{2 \ell +2}-H_{\ell +1}$. Notice also that 
\begin{eqnarray*}
\hat{\phi}_{\ell}(1) = \int_{0}^{\infty} e^{-t} \hat{\phi}_{\ell}(t) {\rm d} t = \sum_{k=0}^{\ell} \left(e^{-t\sum_{i=1}^{k} Y_{i}} - e^{-t\sum_{i=1}^{k+1} Y_{i}} \right) + e^{-t\sum_{i=1}^{\ell+1} Y_{i}}.
\end{eqnarray*}

\noindent Therefore, a simple but tedious computation implies that $
\mathbb{E}[\hat{\phi}_{\ell}(1)] = (\ell +1)(H_{2 \ell +2}-H_{\ell +1})$.
\end{proof}

\subsection{Fragmentation trees} \label{fragmentation}

In this section, we study family trees induced by fragmentation processes. These processes were introduced in \cite{Kol1941}; see also \cite{Beb2006} and \cite{Svante2008} for general background and further references. Fix $b \geq 2$ and consider a random vector $\mathbf{V} = (V_{1}, \dots, V_{b})$;  this is the so-called dislocation law. Assume that $0 \leq V_{i} < 1$ a.s., for $i =1, \dots, b$, and $\sum_{i=1}^{b} V_{i} = 1$, i.e.\ $\mathbf{V}$ belongs to the standard simplex. 

We then describe the construction of the fragmentation tree. Start with a vertex (the root) with mass $x_{0}$. This vertex has $b$ children with masses $x_{0}V_{1}, \dots, x_{0}V_{b}$, i.e.\ we break $x_{0}$ into $b$ pieces with masses driven by the dislocation law $\mathbf{V}$. Consider a threshold $x_{1} \in (0, x_{0}]$, and continue recursively with each vertex that has mass larger than $x_{1}$, using new (independent) copies of $\mathbf{V}$ each time. The process terminates a.s.\ after a finite number of steps, leaving a finite set of vertices (or fragments) with masses smaller than $x_{1}$. We regard the vertices of mass larger than $x_{1}$ that occur during this process as the internal vertices and the resulting vertices of mass strictly less than $x_{1}$ as external. Notice that the fragmentation tree depends only on the ratio $x_{0}/x_{1}$, so we denote it by $T_{x_{0}/x_{1}}^{\rm f}$. 

One can relate the fragmentation process to a CMJ-process by regarding a fragment of mass $x$ as born at time $\log(x_{0}/x)$. In other words, a vertex will have $b$ children that are born at times $\xi_{i} = - \log V_{i}$, for $i =1, \dots, b$ (observe that $N=b$ in this case). If $V_{i}=0$, one has that $\xi_{i} = \infty$, meaning that this child is not born, and thus, a vertex has less than $b$ children. Notice also that the life time $\lambda = \infty$. Finally, it is easy to see that the fragmentation tree $T_{x_{0}/x_{1}}^{\rm f}$ is the same as the family tree of this CMJ-process at time $\log(x_{0}/x_{1})$, i.e. $T(\log(x_{0}/x_{1}))$ in the notation of Section \ref{Model}.

By using the characteristic $\phi \equiv 1$, we define the fragmentation tree $T_{n}^{\rm f} := T(\tau_{n}^{\phi})$ of fixed size $n \in \mathbb{N}$ as in Definition \ref{def1}. This means that we choose a threshold $x_{1} > 0$ to be the mass of the $n$-th largest fragment in the process, so that there will be exactly $n$ fragments of size $x_{1}$ (unless there is a tie).  In this case, notice that $\hat{\Xi}(\theta) = \sum_{i=1}^{b} e^{- \theta \xi_{i}} =  \sum_{i=1}^{b} V_{i}^{\theta}$. Then
\begin{eqnarray} \label{eq22}
\hat{\mu}(\theta) = \sum_{i=1}^{b} \mathbb{E}[V_{i}^{\theta}], \hspace*{5mm} \theta \geq 0.
\end{eqnarray}

\noindent Thus, one concludes that the {\sl Malthusian parameter} is $\alpha=1$. This has the consequence that the martingale $W^{\psi}$ (with $p_{n} \equiv 1$), in Section \ref{ProofM},  is constant $1$. Then one has that its limit $W^{\psi}(\infty) \equiv 1$. 
Furthermore, one also has that $Var (\hat{\Xi}(\theta) ) < \infty$, for $\theta \geq 0$. Thus, the conditions ({\bf A1})–({\bf A7}) are satisfied and Theorem \ref{Teo2} implies the following result. 

\begin{corollary}
In the supercritical regime, we have that
$\lim_{n \rightarrow \infty} n^{-1}|T_{n}^{\rm f}| = e^{-\frac{c}{ \beta}}$, in probability, where $\beta := \sum_{i=1}^{b} \mathbb{E}[V_{i} \log(1/V_{i})]$.
\end{corollary}

\begin{proof}
The result follows from Theorem \ref{Teo2}. Since 
$\phi \equiv 1$, one only needs to compute $\bar{\mu}(1)$. But notice that (\ref{eq38}) and (\ref{eq22}) imply that $\bar{\mu}(1) = - \hat{\mu}^{\prime}(1) = \beta$.
\end{proof}

\begin{example}[Binary splitting]
Consider $b =2$ and $\mathbf{V} = (V_{1}, V_{2}) =(V_{1}, 1-V_{1})$, where $V_{1}$ is an uniform random variable on $(0,1)$. Thus, at each fragmentation event, the fragment is split into two parts, with uniformly random sizes.  In the corresponding CMJ-process the birth times $\xi_{1}$ and $\xi_{2}$ are exponential random variables with parameter $1$, where one of them determines the other by
$e^{-\xi_{1}} + e^{-\xi_{2}} = 1$. Furthermore, $\hat{\mu}(\theta) = (1+\theta)^{-1}$, for $\theta \geq 0$, and $\bar{\mu}(1) = 1/2$. Finally, note also that there are similarities with the CMJ-process associated with the binary search tree in Section \ref{general}, Table \ref{table1};  the difference is that there $\xi_{1}$ and $\xi_{2}$ are independent, while here they are dependent. 
\end{example}

Finally, it is important to mention that the split trees defined by Devroye \cite{Luc1999} are related to fragmentation trees. A split tree is a $b$-ary tree defined using a number of balls that enter the root and are distributed (randomly and recursively) to the subtrees of the root and further down in the tree according to certain rules that are based on a splitting law $\mathbf{V}$; see \cite{Luc1999} for details. For instance, binary search trees are a particular type of split tree. Percolation on split trees will be studied in another paper.

\subsection{Homogeneous CMJ-trees} \label{Homo}

Let $\Lambda$ be a finite positive measure on $(0, \infty]$ with total mass $b := \Lambda((0, \infty])$ such that $m:= \int_{(0, \infty]} t \Lambda({\rm}dt)$ satisfies
\begin{equation} \label{cond3}
 1 < m < \infty.  \tag{\bf E3}
\end{equation}

\noindent Consider a CMJ-process with birth process $\Xi$ whose intensity measure is given by $\mu({\rm d}t) = {\rm d}t \Lambda((t, \infty])$. Moreover, each individual has lifetime given by a random variable $\lambda$ with distribution $\Lambda(\cdot) /b$. Notice that conditional on the birth time and lifetime, the birth point process of an individual is distributed as a Poisson point process during his life. One can say that the CMJ-process is homogeneous (constant birth rate) and binary (birth occurs singly). Set 
\begin{eqnarray*}
\Psi(\theta) = \theta - \int_{(0, \infty]}(1-e^{-\theta t}) \Lambda({\rm d} t), \hspace*{5mm} \theta \geq 0,
\end{eqnarray*}

\noindent and observe that $\Psi$ is a convex function such that $\Psi(0+) =0$ and $\Psi^{\prime}(0+) = 1-m < 0$. Hence there exists a unique $\alpha > 0$ such that $\Psi(\alpha) = 0$.



In this section, we are interested in study Bernoulli bond-percolation on the family tree of this CMJ-process. More precisely, we consider the characteristic $\phi \equiv 1$ and define the CMJ-tree $T_{n}^{{\rm hom}} := T(\tau^{\phi}_{n})$ as in Definition \ref{def1}. The family tree of this particular CMJ-process has been called splitting tree in \cite{Gei1996}, \cite{Gei1997} and \cite{Lamb2010}. In these works, it has been further studied the case when $\Lambda$ is not necessary finite. But for simplicity we have decided to restrict ourselves to finite measures. Nevertheless, one can apply our results to the general case.

In \cite[Chapter 3]{Rich2011}, the author studied this CMJ-process under neutral rare mutations, where mutations occur at birth with probability $1-p \in [0,1]$. The difference with our setting is that in \cite[Chapter 3]{Rich2011} the probability of mutation (or percolation parameter) does not dependent on the ``size'' of the tree. Therefore, the result in this section may be of independent interest. 

We now check that assumptions $({\bf A1})$-$({\bf A7})$ are satisfied. Clearly, $({\bf A1})$-$({\bf A2})$ are fulfilled by the assumptions made on $\Lambda$. Notice that in this case $({\bf A3})$ is not fulfilled. Nevertheless, \cite[Propostion 2.1; see also (2.1)]{Rich2011} shows that the extinction probability $\mathbb{P}(Z(\infty) < \infty) = 1-\alpha/b < 1$. This has the consequence that the limit of the martingale $W^{\psi}$ (with $p_{n}\equiv 1$), in Section \ref{ProofM}, is strictly positive (i.e.\ $W^{\psi}(\infty) > 0$ a.s.) on the event $\{Z(\infty) = \infty\}$; see \cite[Corollary 4.2]{JaNe}. Thus our results hold conditioned on the event $\{Z(\infty) = \infty\}$ since $({\bf A3})$ is only needed to guarantee that $W^{\psi}(\infty) > 0$ almost surely. On the other hand, 
\begin{eqnarray} \label{eq23}
\hat{\mu}(\theta) = \int_{0}^{\infty} \int_{(t, \infty]} e^{-\theta t} {\rm d} t \Lambda({\rm d} u) = \int_{(0, \infty]} \frac{1-e^{-\theta t}}{\theta}  \Lambda({\rm d} t) = \frac{\theta-\Psi(\theta)}{\theta}, \hspace*{5mm} \theta \geq 0,
\end{eqnarray}

\noindent  which clearly shows that $({\bf A4})$ is satisfied. Moreover, the {\sl Malthusian parameter} is given by $\alpha > 0$. Notice also that the Campbell's formula implies that  
\begin{eqnarray*}
Var (\hat{\Xi}(\theta)) = \int_{0}^{\infty} \int_{(t, \infty]} e^{-2\theta t} {\rm d} t \Lambda({\rm d} u) =  \frac{2\theta-\Psi(2\theta)}{2\theta} < \infty, \hspace*{5mm} \theta \geq 0,
\end{eqnarray*}

\noindent which implies $({\bf A5})$. Finally, $({\bf A6})$-$({\bf A7})$ follows immediately since $\phi \equiv 1$. 
Therefore, Theorem \ref{Teo2} implies the following result. 

\begin{corollary}
In the supercritical regime and under $(\ref{cond3})$, we have conditional on the event $\{Z(\infty) = \infty\}$ that $\lim_{n \rightarrow \infty} n^{-1} |T_{n}^{{\rm hom}}| = e^{-\frac{c}{\Psi^{\prime}(\alpha)}}$, in probability. 
\end{corollary}

\begin{proof}
The result follows from Theorem \ref{Teo2} by computing $\bar{\mu}(\alpha)$ since $\phi \equiv 1$. This follows from (\ref{eq27}),
\begin{eqnarray*}
\bar{\mu}(\alpha) = \int_{0}^{\infty} t e^{- \alpha t} \mu({\rm d} t) =  \int_{0}^{\infty} \int_{(t, \infty]} t e^{-\alpha t} {\rm d} t \Lambda({\rm d} u) = \int_{(0, \infty]} \left(-\frac{t e^{-\alpha t}}{\alpha} + \frac{1-e^{-\alpha t}}{\alpha^{2}} \right) \Lambda({\rm d} t) = \frac{\Psi^{\prime}(\alpha)}{\alpha}.
\end{eqnarray*}
\end{proof}

\section{Appendix: Proofs of Lemmas \ref{lemma1} and \ref{lemma2}} \label{prel} 

In this section we establish some general results on the long time behavior of the CMJ-process with neutral mutations that may be of independent interest. It will be helpful to write $p$ rather than $p_{n}$, omitting the integer $n$ from the notation. To be more precise, we consider that the percolation parameter is a real number $p \in [0,1]$ and study the behavior of the $\phi$-counted clonal process $Z_{\varnothing}^{(p), \phi} = (Z_{\varnothing}^{(p), \phi}(t), t \geq 0)$ as $p \rightarrow 1$ and $t \rightarrow \infty$. Recall that  $\Xi^{(p)}$ denotes the generic birth process of the clonal CMJ-process whose intensity measure $\mu^{(p)}$ defined in (\ref{eq3}).

Notice that $({\bf A4})$ implies that there exists $p^{\ast} \in (0,1)$ such that for $p \in [p^{\ast},1]$ there is $\alpha_{p} > 0$ (the {\sl Malthusian parameter} of $\mu^{(p)}$) such that $\hat{\mu}^{(p)}(\alpha_{p}) =1$. This implies that $p \mathbb{E}[N] > 1$ (i.e. the clonal process is supercritical). Since $\alpha_{p} \rightarrow \alpha$ as $p \rightarrow 1$, we choose $p^{\ast}$ such that $0 < \theta_{1} < \alpha_{p^{\ast}}$ where $\theta_{1}$ satisfies $({\bf A4})$-$({\bf A5})$. Furthermore, we also consider $p^{\ast}$ such that 
$0 < \theta_{2} < \alpha_{p^{\ast}}$ and $0 < \theta_{3} < 2\alpha_{p^{\ast}}$, where $\theta_{2}$ and $\theta_{3}$ satisfy $({\bf A6})$ and $({\bf A7})$, respectively.

We start by recalling some well-known results about the moments of $Z_{\varnothing}^{(p), \phi}$. For $k \in \mathbb{N} \cup \{ 0 \}$, we denote by  $\nu^{\ast k}$ the $k$-fold convolution of a measure $\nu$ on $[0, \infty)$ (here $\nu^{\ast 0}$ is a unit point mass at $0$). 

\begin{theorem} \label{Teo1}
Assume that conditions $({\bf A1})$-$({\bf A4})$ are fulfilled. For a characteristic $\phi$ that may depend on $p$, we have that 
\begin{eqnarray*}
\mathbb{E}\left[ Z_{\varnothing}^{(p), \phi}(t) \right] = \int_{0}^{t} \mathbb{E}[\phi(t-s)] \sum_{k=0}^{\infty} (\mu^{(p)})^{\ast k} ({\rm d} s), \hspace*{5mm} t \geq 0.
\end{eqnarray*}

\noindent Furthermore, if $\mathbb{E}[\phi(t)]$ is bounded on finite intervals,
\begin{eqnarray*}
Var \left( Z_{\varnothing}^{(p), \phi}(t) \right) = \int_{0}^{t} h_{\varnothing}(t-s) \sum_{k=0}^{\infty} (\mu^{(p)})^{\ast k} ({\rm d} s), \hspace*{5mm} t \geq 0,
\end{eqnarray*}

\noindent with 
\begin{eqnarray}
h_{\varnothing}^{(p)}(t) = Var \left(\phi_{\varnothing}(t) + \int_{0}^{t} \mathbb{E}\left[ Z_{\varnothing}^{(p), \phi}(t-s) \right] \Xi_{\varnothing}^{(p)}({\rm d} s) \right), \hspace*{5mm} t \geq 0, \label{eq1}
\end{eqnarray}

\noindent and where $(\Xi_{\varnothing}^{(p)}, \phi_{\varnothing})$ is the birth process and weight associated to the progenitor of the population.
\end{theorem}

\begin{proof}
The first claim is a consequence of \cite[Theorem 3.1]{JaNe}. Notice that $({\bf A1})$ and $({\bf A4})$ imply that the measure $\sum_{k=0}^{\infty} (\mu^{(p)})^{\ast k} ({\rm d} t)$ is finite. Moreover, $\mathbb{E}[ Z_{\varnothing}^{(p), \phi}(t)] < \infty$, for $t \geq 0$, since $\mathbb{E}[\phi(t)]$ is bounded on finite intervals. Therefore, the second claim follows from \cite[Theorem 3.2]{JaNe}.
\end{proof}

We next provide an improvement of the results of \cite[Theorem 6.9.2]{Ja} and \cite[Theorem 3.5]{JaNe}  on the asymptotic behaviour of the first and second moment of the clonal CMJ-process. Recall that we write $W^{(p), \phi}_{\varnothing} = (W^{(p), \phi}_{\varnothing}(t), t \geq 0)$ for the process given by $W^{(p), \phi}_{\varnothing}(t) := e^{-t\alpha_{p} } Z^{(p), \phi}_{\varnothing}(t)$, for $t \geq 0$. Set
\begin{eqnarray*}
m_{t}^{(p), \phi} = \mathbb{E} [ W^{(p), \phi}_{\varnothing}(t) ], \hspace*{5mm} t \geq 0.
\end{eqnarray*}

\noindent Recall also that for $p\equiv1$ we sometimes remove the superscript $(p)$ and the subscript $\varnothing$ from the previous notations. That is, we write $W^{\phi} = (W^{\phi}(t), t \geq 0)$ for the processes given by  $W^{\phi}(t) := e^{-t\alpha } Z^{\phi}(t)$ and $m_{t}^{\phi} = \mathbb{E} [ W^{\phi}(t) ]$.

\begin{proposition} \label{Pro1}
Assume that conditions $({\bf A1})$-$({\bf A4})$ are fulfilled. 
\begin{itemize}
\item[(i)] For any characteristic $\phi$ that does not depend on $p$ and that satisfies $({\bf A6})$, we have that
\begin{eqnarray*}
 \lim_{t \rightarrow \infty} \sup_{p \in [p^{\ast}, 1]} \left| m_{t}^{(p), \phi} - m_{\infty}^{(p), \phi} \right | = 0,
\end{eqnarray*}  

\noindent where $m_{\infty}^{(p), \phi}$ is defined in (\ref{eq39}). In particular, $ \sup_{t \geq 0} \sup_{p \in [p^{\ast},1]} m_{t}^{(p), \phi} < \infty$.

\item[(ii)] For the characteristic $\psi$ defined in (\ref{eq10}), we have that 
\begin{eqnarray*}
 \lim_{t \rightarrow \infty} \sup_{p \in [p^{\ast}, 1]} \left| m_{t}^{(p), \psi} - m_{\infty}^{(p), \psi} \right | = 0.
\end{eqnarray*}  

\noindent In particular, $ \sup_{t \geq 0} \sup_{p \in [p^{\ast},1]} m_{t}^{(p), \psi} < \infty$.
\end{itemize}
\end{proposition}

\begin{proof}
We only prove (i). The proofs of (ii) follows from exactly same argument and Lemma \ref{lemma4} (i). Recall from Remark \ref{rem4}  that $\mu_{\alpha}^{(p)}(dt) := e^{-t\alpha_{p}} \mu^{(p)}(dt)$, for $p \in [p^{\ast}, 1]$, is a probability measure on $(0, \infty)$. Theorem \ref{Teo1} implies that 
\begin{eqnarray*}
m_{t}^{(p), \phi} = \int_{0}^{t} e^{- (t-s) \alpha_{p}} \mathbb{E}[\phi(t-s)] \sum_{k=0}^{\infty} (\mu_{\alpha}^{(p)})^{\ast k} ({\rm d} s), \hspace*{5mm} t \geq 0,
\end{eqnarray*}

\noindent where we have used that
\begin{eqnarray*}
e^{- s\alpha_{p}}\sum_{k=0}^{\infty} (\mu^{(p)})^{\ast k} ({\rm d} s) =  \sum_{k=0}^{\infty} (\mu_{\alpha}^{(p)})^{\ast k} ({\rm d} s).
\end{eqnarray*}

\noindent From $({\bf A6})$, we deduce that the family of functions $t \mapsto e^{-t\alpha_{p}} \mathbb{E}[\phi(t)]$, for $p \in [p^{\ast}, 1]$, it is uniformly directly Riemann integrable (see \cite[Definition 2.8]{Bor2013}). Furthermore, $({\bf A4})$ implies that the family of probability measures $\{ \mu_{\alpha}^{(p)}: p \in [p^{\ast}, 1] \}$ is weakly compact (treated as a set of measure on $(0,\infty)$), and uniformly integrable, that is, 
\begin{eqnarray*}
\lim_{a \rightarrow \infty} \sup_{p \in [p^{\ast}, 1]} \int_{[a, \infty)} t \mu_{\alpha}^{(p)}({\rm d}t) = 0.
\end{eqnarray*}

\noindent Therefore, the point (i) is a consequence of the uniform version of the key renewal theorem \cite[Theorem 2.12]{Bor2013} since $\mu$ satisfies $({\bf A2})$. The second claim in (i) follows from the definition of $m_{\infty}^{(p), \phi}$ by noticing that
\begin{eqnarray*}
\alpha_{p^{\ast}} \leq \alpha_{p} \leq \alpha, \hspace*{3mm}  \bar{\mu}^{(p)}(\alpha_{p}) \geq p^{\ast}\bar{\mu}^{(1)}(\alpha) \hspace*{3mm} \text{and} \hspace*{3mm} \mathbb{E}[\hat{\phi}(\alpha_{p})] \leq \frac{\alpha}{ \alpha_{p^{\ast}}} \mathbb{E}[\hat{\phi}(\alpha_{p^{\ast}})];
\end{eqnarray*}

\noindent one also needs to recall that $\bar{\mu}^{(1)}(\alpha) >0$ by Remark \ref{rem2}. 
\end{proof}

By Proposition \ref{Pro1}, 
\begin{eqnarray*}
v_{t}^{(p), \phi} = Var ( W^{(p), \phi}_{\varnothing}(t) ), \hspace*{5mm} t \geq 0,
\end{eqnarray*}

\noindent is well defined. For $p \equiv 1$, we sometimes write  $v_{t}^{\phi} = Var ( W^{\phi}(t) )$, for $t \geq 0$. 

\begin{proposition} \label{Pro2}
Assume that conditions $({\bf A1})$-$({\bf A5})$ are fulfilled.
\begin{itemize}
\item[(i)] For any characteristic $\phi$ that does not depend on $p$ and that satisfies $({\bf A6})$-$({\bf A7})$, we have that
\begin{eqnarray*}
 \lim_{t \rightarrow \infty} \sup_{p \in [p^{\ast}, 1]} \left|v_{t}^{(p), \phi} -v_{\infty}^{(p), \phi} \right | = 0,
\end{eqnarray*}

\noindent where
\begin{eqnarray*}
v_{\infty}^{(p), \phi} := \left( m_{\infty}^{(p), \phi} \right)^{2} \frac{Var(\hat{\Xi}^{(p)}(\alpha_{p}))}{1-\hat{\mu}^{(p)}(2 \alpha_{p})}.
\end{eqnarray*}

\item[(ii)] For the characteristic $\psi$ defined in (\ref{eq10}), we have that 
\begin{eqnarray*}
 \lim_{t \rightarrow \infty} \sup_{p \in [p^{\ast}, 1]} \left|v_{t}^{(p), \psi} -v_{\infty}^{(p), \psi} \right | = 0,
\end{eqnarray*}

\noindent where $ v_{\infty}^{(p), \psi} :=  Var(\hat{\Xi}^{(p)}(\alpha_{p})) (1-\hat{\mu}^{(p)}(2 \alpha_{p}))^{-1}$. 

\item[(iii)] For the characteristic $\phi^{\prime} = \phi + m_{\infty}^{(p), \phi} \psi$, we have that 
\begin{eqnarray*}
 \lim_{t \rightarrow \infty} \sup_{p \in [p^{\ast}, 1]} \left|v_{t}^{(p), \phi^{\prime}} -v_{\infty}^{(p), \phi^{\prime}} \right | = 0,
\end{eqnarray*}

\noindent where $ v_{\infty}^{(p), \phi^{\prime}} = 4 v_{\infty}^{(p), \phi}$. 
\end{itemize}
\end{proposition}

\begin{proof}
We only prove (i). The proof of (ii) and (iii) follows similarly by using Lemma \ref{lemma4} and Remark \ref{rem3}. Theorem \ref{Teo1} implies that
\begin{eqnarray*}
v_{t}^{(p), \phi} = \int_{0}^{t} e^{-2(t-s)\alpha_{p}} h^{(p)}_{\varnothing}(t-s) e^{-2s\alpha_{p}} \sum_{k=0}^{\infty} (\mu^{(p)})^{\ast k} ({\rm d} s), \hspace*{5mm} t \geq 0,
\end{eqnarray*}

\noindent where the function $h^{(p)}_{\varnothing}$ is defined in (\ref{eq1}). First, notice the following identity
\begin{eqnarray} \label{eq15}
\int_{0}^{\infty} e^{-2s \alpha_{p} }  \sum_{k=0}^{\infty} (\mu^{(p)})^{\ast k} ({\rm d} s) = \sum_{k=0}^{\infty} \hat{\mu}^{(p)}(2 \alpha_{p})^{k} = \frac{1}{1-\hat{\mu}^{(p)}(2 \alpha_{p})} < \infty,
\end{eqnarray}

\noindent since $\hat{\mu}^{(p)}(2 \alpha_{p}) < \hat{\mu}^{(p)}(\alpha_{p}) =1$. Then, the triangle inequality implies that 
\begin{align} \label{eq5}
\left|v_{t}^{(p), \phi} -v_{\infty}^{(p), \phi} \right | & \leq \int_{0}^{t}e^{-2(t-s)\alpha_{p}} Var(\phi_{\varnothing}(t-s)) e^{ -2s\alpha_{p}} \sum_{k=0}^{\infty} (\mu^{(p)})^{\ast k} ({\rm d} s) \nonumber  \\
& \hspace*{2mm} + 2\int_{0}^{t}e^{-(t-s)\alpha_{p}} \left| Cov \left(\phi_{\varnothing}(t-s),   \int_{0}^{t-s} m_{t-s-u}^{(p), \phi} e^{-u \alpha_{p}} \Xi^{(p)}_{\varnothing}({\rm d} u)  \right)\right|  e^{ -2s\alpha_{p}} \sum_{k=0}^{\infty} (\mu^{(p)})^{\ast k} ({\rm d} s) \nonumber \\
& \hspace*{2mm} + \int_{0}^{\infty} g^{(p)}(t-s) e^{-2s\alpha_{p}} \sum_{k=0}^{\infty} (\mu^{(p)})^{\ast k} ({\rm d} s), 
\end{align}

\noindent  where $(\Xi_{\varnothing}^{(p)}, \phi_{\varnothing})$ is the birth process and weight associated to the progenitor of the population and 
\begin{eqnarray*}
g^{(p)}(t-s) = \left | \mathds{1}_{\{ s\in [0,t] \}}  Var \left(  \int_{0}^{t-s} m_{t-s-u}^{(p), \phi} e^{-u\alpha_{p}} \Xi^{(p)}_{\varnothing}({\rm d} u) \right) - \left( m_{\infty}^{(p), \phi} \right)^{2} 
Var(\hat{\Xi}^{(p)}_{\varnothing}(\alpha_{p}))
 \right|.
\end{eqnarray*}

\noindent Denote by $I_{1}^{(p)}(t)$, $I_{2}^{(p)}(t)$ and $I_{3}^{(p)}(t)$ the first, the second and the third term on the right-hand side of (\ref{eq5}), respectively. Then the claim in Proposition \ref{Pro2} follows by showing that
\begin{itemize}
\item[(a)] $\lim_{t \rightarrow \infty} \sup_{p \in [p^{\ast}, 1]} I_{1}^{(p)}(t) = 0$,
\item[(b)] $\lim_{t \rightarrow \infty} \sup_{p \in [p^{\ast}, 1]} I_{2}^{(p)}(t) = 0$ and
\item[(c)] $\lim_{t \rightarrow \infty} \sup_{p \in [p^{\ast}, 1]} I_{3}^{(p)}(t) = 0$.  
\end{itemize}

We start by showing (a). By assumption $({\bf A7})$,
\begin{eqnarray} \label{eq6}
\lim_{t \rightarrow \infty} \sup_{p \in [p^{\ast}, 1]} e^{-2t\alpha_{p}} Var(\phi(t)) \leq \lim_{t \rightarrow \infty} e^{-2t\alpha_{p^{\ast}}} Var(\phi(t)) = 0.
\end{eqnarray}

\noindent Then (\ref{eq15}) and the dominated convergence theorem prove point (a). 

Next we show (b). The Cauchy-Schwarz's inequality and Proposition \ref{Pro1} imply that
\begin{eqnarray*}
\left| Cov \left(\phi_{\varnothing}(t),   \int_{0}^{t} m_{t-u}^{(p), \phi} e^{-u \alpha_{p}} \Xi^{(p)}_{\varnothing}({\rm d} u)  \right)\right|   & \leq & (Var(\phi_{\varnothing}(t)))^{\frac{1}{2}} \left( Var\left( \int_{0}^{t} m_{t-u}^{(p), \phi} e^{-u \alpha_{p}} \Xi^{(p)}_{\varnothing}({\rm d} u)\right) \right)^{\frac{1}{2}}  \\
& \leq &  \left(Var(\phi_{\varnothing}(t))  Var\left( \int_{0}^{t} e^{-u \alpha_{p}} \Xi^{(p)}_{\varnothing}({\rm d} u)\right)  \right)^{\frac{1}{2}} \sup_{t \geq 0} \sup_{p \in [p^{\ast},1]} m_{t}^{(p), \phi}.
\end{eqnarray*}

\noindent Since $e^{-u\alpha_{p} } \Xi^{(p)}({\rm d} u)$ is dominated by $e^{-u\alpha_{p^{\ast}}} \Xi({\rm d} u)$-$({\bf A4})$, $({\bf A5})$ and (\ref{eq6}) allow us to deduce  that 
\begin{eqnarray*}
\lim_{t \rightarrow \infty} \sup_{p \in [p^{\ast}, 1]}  e^{-t\alpha_{p}} \left| Cov \left(\phi_{\varnothing}(t),   \int_{0}^{t} m_{t-u}^{(p), \phi} e^{-u \alpha_{p}} \Xi^{(p)}_{\varnothing}({\rm d} u)  \right)\right|  = 0.
\end{eqnarray*}

\noindent Hence, an application of the dominated convergence theorem shows (b).

Finally, we prove (c). We show that
\begin{eqnarray} \label{eq4}
 \lim_{t \rightarrow \infty} \sup_{p \in [p^{\ast}, 1]} g^{(p)}(t-s) = 0,
\end{eqnarray}

\noindent which together with an application of the dominated convergence theorem implies (c). Observe that
\begin{align} \label{eq2}
g^{(p)}(t-s) & \leq g_{1}^{(p)}(t-s) + g_{2}^{(p)}(t-s) + g_{3}^{(p)}(t-s)  \nonumber \\
& \hspace*{10mm} + g_{4}^{(p)}(t-s)  + \mathds{1}_{\{ s\in (t,\infty) \}} \left( m_{\infty}^{(p), \phi} \right)^{2} 
Var(\hat{\Xi}^{(p)}(\alpha_{p})),
\end{align}

\noindent where 
\begin{eqnarray*}
g_{1}^{(p)}(t-s) = \mathds{1}_{\{ s\in [0,t] \}} Var\left( \int_{0}^{t-s} \left( m_{t-s-u}^{(p), \phi} - m_{\infty}^{(p), \phi} \right) e^{-u\alpha_{p}} \Xi^{(p)}({\rm d} u)  \right),
\end{eqnarray*}

\begin{eqnarray*}
g_{2}^{(p)}(t-s) = 2 \mathds{1}_{\{ s\in [0,t] \}}\left|  Cov\left(  \int_{0}^{t-s} \left( m_{t-s-u}^{(p), \phi} - m_{\infty}^{(p), \phi} \right) e^{-u\alpha_{p}} \Xi^{(p)}({\rm d} u),   m_{\infty}^{(p), \phi} \int_{0}^{t-s}  e^{-u\alpha_{p}} \Xi^{(p)}({\rm d} u)  \right) \right|,
\end{eqnarray*}

\begin{eqnarray*}
g_{3}^{(p)}(t-s) = 2 \left( m_{\infty}^{(p), \phi}  \right)^{2} \mathds{1}_{\{ s\in [0,t] \}}\left|  Cov\left(  \hat{\Xi}^{(p)}(\alpha_{p}),   \int_{t-s}^{\infty}  e^{-u\alpha_{p}} \Xi^{(p)}({\rm d} u)  \right) \right|,
\end{eqnarray*}

\noindent and
\begin{eqnarray*}
g_{4}^{(p)}(t-s) = \mathds{1}_{\{ s\in [0,t] \}}  \left( m_{\infty}^{(p), \phi}  \right)^{2}  Var\left( \int_{t-s}^{\infty}  e^{-u\alpha_{p}} \Xi^{(p)}({\rm d} u)  \right).
\end{eqnarray*}

\noindent Proposition \ref{Pro1} and $({\bf A5})$ imply that
\begin{eqnarray} \label{eq7}
\lim_{t \rightarrow \infty} \sup_{p \in [p^{\ast}, 1]} g_{4}^{(p)}(t-s) = 0 \hspace*{5mm} \text{and} \hspace*{5mm} \lim_{t \rightarrow \infty} \sup_{p \in [p^{\ast}, 1]} \mathds{1}_{\{ s\in (t,\infty) \}} \left( m_{\infty}^{(p), \phi} \right)^{2} 
Var(\hat{\Xi}^{(p)}(\alpha_{p})) = 0.
\end{eqnarray}

\noindent Furthermore, Proposition \ref{Pro1} implies that 
\begin{eqnarray*}
 \sup_{t \geq 0} \sup_{p \in [p^{\ast}, 1]} \left| m_{t}^{(p), \phi} - m_{\infty}^{(p), \phi} \right | < \infty.
\end{eqnarray*}

\noindent Recall that $e^{-s \alpha_{p} } \Xi^{(p)}({\rm d} s)$ is dominated by $e^{-s\alpha_{p^{\ast}}} \Xi({\rm d} s)$ and that $\int_{0}^{\infty}  e^{-s\alpha_{p^{\ast}} } \Xi({\rm d} s) < \infty$, by $({\bf A4})$. Thus  the dominated convergence theorem shows that
\begin{eqnarray*}  
 \lim_{t \rightarrow \infty} \sup_{p \in [p^{\ast}, 1]} \left|  \mathds{1}_{\{ s\in [0,t] \}}  \int_{0}^{t-s} \left( m_{t-s-u}^{(p), \phi} - m_{\infty}^{(p), \phi} \right) e^{-u\alpha_{p}} \Xi^{(p)}({\rm d} u)   \right | = 0,
\end{eqnarray*}

\noindent almost surely. This implies together with the dominated convergence theorem once again that 
\begin{eqnarray} \label{eq8}
\lim_{t \rightarrow \infty} \sup_{p \in [p^{\ast}, 1]} g_{1}^{(p)}(t-s) = 0.
\end{eqnarray}

\noindent Similarly, one can deduce that 
\begin{eqnarray} \label{eq9}
\lim_{t \rightarrow \infty} \sup_{p \in [p^{\ast}, 1]} g_{2}^{(p)}(t-s) = 0 \hspace*{5mm} \text{and} \hspace*{5mm} \lim_{t \rightarrow \infty} \sup_{p \in [p^{\ast}, 1]} g_{3}^{(p)}(t-s) = 0.
\end{eqnarray}

\noindent Finally, our claim in (\ref{eq4}) follows by combining (\ref{eq2}), (\ref{eq7}), (\ref{eq8}) and (\ref{eq9}). 
\end{proof}

We have now all the ingredients to prove Lemma \ref{lemma1}. 

\begin{proof}[Proof of Lemma \ref{lemma1}]
Notice that 
\begin{align} \label{eq30}
& \mathbb{E} \left[ \sup_{s\geq t} \left| W^{(p), \psi}_{\varnothing}(s) - W^{(p), \psi}_{\varnothing}(\infty)  \right|^{2} \right] \nonumber \\
& \hspace*{20mm} \leq 2 \mathbb{E} \left[ \left| W^{(p), \psi}_{\varnothing}(t) - W^{(p), \psi}_{\varnothing}(\infty)  \right|^{2} \right] + 2 \mathbb{E} \left[ \sup_{s\geq t} \left| W^{(p), \psi}_{\varnothing}(s) - W^{(p), \psi}_{\varnothing}(t)  \right|^{2} \right],
\end{align}

\noindent for $t \geq 0$. On the one hand, from properties of square-integrable martingales, we obtain that
\begin{eqnarray} \label{eq31}
\mathbb{E}\left[ \left| W^{(p), \psi}_{\varnothing}(t) - W^{(p), \psi}_{\varnothing}(\infty)  \right|^{2} \right] & = & \mathbb{E}\left[ W^{(p), \psi}_{\varnothing}(\infty)^{2} \right] - \mathbb{E}\left[  W^{(p), \psi}_{\varnothing}(t)^{2} \right], \hspace*{5mm} t \geq 0.
\end{eqnarray}

\noindent On the other hand, by Doob's inequality 
\begin{eqnarray} \label{eq32}
\mathbb{E} \left[ \sup_{s\geq t} \left| W^{(p), \psi}_{\varnothing}(s) - W^{(p), \psi}_{\varnothing}(t)  \right|^{2} \right] \leq 4 \mathbb{E}\left[  W^{(p), \psi}_{\varnothing}(\infty)^{2} \right] - 4\mathbb{E}\left[  W^{(p), \psi}_{\varnothing}(t)^{2} \right], \hspace*{5mm} t \geq 0.
\end{eqnarray}

\noindent By combining (\ref{eq30}), (\ref{eq31}) and (\ref{eq32}), we deduce that
\begin{eqnarray*}
\mathbb{E}\left[ \sup_{s \geq t} \left| W^{(p), \psi}_{\varnothing}(s) - W^{(p),\psi}_{\varnothing}(\infty)  \right|^{2} \right] & \leq & 10 \left( v_{\infty}^{(p), \psi} - v_{t}^{(p), \psi} \right) + 10 \left( \left( m_{\infty}^{(p), \psi} \right)^{2}- \left(  m_{t}^{(p), \psi} \right)^{2} \right) \\
& = & 10 \left( v_{\infty}^{(p), \psi} - v_{t}^{(p), \psi} \right)
\end{eqnarray*}

\noindent since $W^{(p), \psi}_{\varnothing}$ is a martingale. Therefore, the first statement follows from Propositions \ref{Pro2}.

We turn our attention to the second claim. Observe that
\begin{align} \label{eq11}
& \mathbb{E} \left[  \left| W^{(p), \phi}_{\varnothing}(t) - m_{\infty}^{(p), \phi}W^{(p), \psi}_{\varnothing}(\infty)  \right|^{2} \right] \nonumber \\
& \hspace*{5mm} \leq 2 \mathbb{E} \left[  \left| W^{(p), \phi}_{\varnothing}(t) - m_{\infty}^{(p), \phi}W^{(p), \psi}_{\varnothing}(t)  \right|^{2} \right] + 2 \left(m_{\infty}^{(p), \phi} \right)^{2} \mathbb{E} \left[  \left| W^{(p), \psi}_{\varnothing}(t) - W^{(p), \psi}_{\varnothing}(\infty)  \right|^{2} \right],
\end{align}

\noindent for $t \geq 0$. It follows from the first part that 
\begin{eqnarray*}
\lim_{t \rightarrow \infty} \sup_{p \in [p^{\ast},1]} \mathbb{E} \left[  \left| W^{(p), \psi}_{\varnothing}(t) - W^{(p), \psi}_{\varnothing}(\infty)  \right|^{2} \right] = 0.
\end{eqnarray*}

\noindent In order to conclude, it is enough to show that the first term on the right-hand side of (\ref{eq11}) tends to $0$ uniformly on $p \in [p^{\ast},1]$ as $t \rightarrow \infty$. By Proposition \ref{Pro1}, this is equivalent to show that 
\begin{eqnarray} \label{eq12}
\lim_{t \rightarrow \infty} \sup_{p \in [p^{\ast},1]} Var \left( W^{(p), \phi}_{\varnothing}(t) - m_{\infty}^{(p), \phi}W^{(p), \psi}_{\varnothing}(t) \right) = 0
\end{eqnarray}

\noindent since $W^{(p), \phi}_{\varnothing}(t) + m_{\infty}^{(p), \phi}W^{(p), \psi}_{\varnothing}(t) = W^{(p), \phi^{\prime}}_{\varnothing}(t)$, where 
$ \phi^{\prime}(t) = \phi(t) +  m_{\infty}^{(p), \phi} \psi(t)$. But (\ref{eq12}) follows from Proposition \ref{Pro2}, Remark \ref{rem3} and the identity
\begin{align*}
& Var \left( W^{(p), \phi}_{\varnothing}(t) - m_{\infty}^{(p), \phi}W^{(p), \psi}_{\varnothing}(t) \right) \\
& \hspace*{5mm} = 2 Var \left( W^{(p), \phi}_{\varnothing}(t) \right) + 2 Var \left( m_{\infty}^{(p), \phi}W^{(p), \psi}_{\varnothing}(t) \right) - Var \left( W^{(p), \phi}_{\varnothing}(t) + m_{\infty}^{(p), \phi}W^{(p), \psi}_{\varnothing}(t) \right).
\end{align*}
\end{proof}

Finally, we conclude this section with the proof of Lemma \ref{lemma2}. The idea of the proof is similar to that of \cite[Lemma 3]{BeU2015}.

\begin{proof}[Proof of Lemma \ref{lemma2}]
We first prove that the double limit 
\begin{eqnarray} \label{eq19}
 \lim_{p \rightarrow 1, t \rightarrow \infty} W_{\varnothing}^{(p), \phi}(t) \hspace*{5mm} \text{exists in} \, \, L_{2}(\mathbb{P}).
\end{eqnarray}

\noindent Denote the $L_{2}(\mathbb{P})$-norm by $\parallel \cdot \parallel_{2}$. We claim that 
\begin{eqnarray} \label{eq13}
 \lim_{p \rightarrow 1}   \parallel W^{(p), \phi}_{\varnothing}(t) - W^{ \phi}(t)  \parallel_{2} \, \,  = 0, \hspace*{5mm} \text{for} \, \, t \geq 0. 
\end{eqnarray}

\noindent Recall that at each birth event, independently of all the other individuals, the newborn is a clone of its parent with probability $p$ or a mutant with probability $1-p$. Then, it should be plain from the thinning property of point measure processes that the birth process of the mutant children of the ancestor $\varnothing$, denoted by $\Xi_{\varnothing}^{(p), \text{m}}$, has intensity measure given $(1-p) \mu ({\rm d}t)$. Furthermore, the later is independent of the birth process of the clonal children of the ancestor described in Section \ref{Main}.  Let $Z^{(p), \text{m}} = (Z^{(p), \text{m}}(t), t\geq 0 )$ be the process that counts the number of mutants that have been born up time $t \geq 0$ and define
\begin{eqnarray*}
b^{(p)}_{1} = \inf \{t \geq 0:  Z^{(p), \text{m}}(t) > 0  \}
\end{eqnarray*}

\noindent the first birth time of a mutant. Plainly, $\lim_{p \rightarrow 1} b^{(p)}_{1} = \infty$ in probability, and the probability of the event $\{ t \geq b^{(p)}_{1} \}$ can be made as small as we wish by choosing $p$ sufficiently close to $1$. On the one hand,  as $ Z^{(p), \phi}_{\varnothing}(t) \leq Z^{ \phi}(t)$, we have 
\begin{eqnarray*}
\mathbb{E} \left[  \left| W^{(p), \phi}_{\varnothing}(t) -W^{ \phi}(t) \right|^{2}, t \geq b^{(p)}_{1} \right] \leq (e^{2(\alpha-\alpha_{p}) t} + 1 ) \mathbb{E} \left[  \left| W^{ \phi}(t) \right|^{2}, t \geq b^{(p)}_{1} \right]
\end{eqnarray*}

\noindent and the right-hand side goes to $0$ as $p \rightarrow 1$. On the other hand, on the event $\{ t < b^{(p)}_{1} \}$, we have $ Z^{(p), \phi}_{\varnothing}(t) = Z^{ \phi}(t)$ and hence $ W^{(p), \phi}_{\varnothing}(t) = e^{(\alpha-\alpha_{p}) t} W^{ \phi}(t)$. This yields
\begin{eqnarray*}
\mathbb{E} \left[  \left| W^{(p), \phi}_{\varnothing}(t_{\varepsilon}) -W^{ \phi}(t) \right|^{2}, t < b^{(p)}_{1} \right] \leq (e^{(\alpha-\alpha_{p}) t} - 1 ) \mathbb{E} \left[  \left| W^{ \phi}(t) \right|^{2} \right]
\end{eqnarray*}

\noindent and the right-hand side goes to $0$ as $p \rightarrow 1$. This establishes the convergence in (\ref{eq13}).

Let $\varepsilon > 0$ be arbitrary. By Lemma \ref{lemma1}, we can find $t_{\varepsilon} > 0$ such that
\begin{eqnarray} \label{eq14}
 \sup_{p \in [p^{\ast},1]} \parallel W^{(p), \phi}_{\varnothing}(s) - m_{\infty}^{(p), \phi}W^{(p), \psi}_{\varnothing}(\infty) \parallel_{2} \, \, \leq \frac{\varepsilon}{6} \hspace*{5mm} \text{for all} \hspace*{5mm} s \geq t_{\varepsilon}.
\end{eqnarray}

\noindent By  using (\ref{eq13}), one obtains that for $s \geq t_{\varepsilon}$ there is $ \delta_{\varepsilon} > 0$ such that $1-p < \delta_{\varepsilon} $ (for $p \in [p^{\ast}, 1)$) implies that
\begin{eqnarray} \label{eq21}
  \parallel W^{(p), \phi}_{\varnothing}(s) - W^{ \phi}(s) \parallel_{2} \, \, \leq \frac{\varepsilon}{6}.
\end{eqnarray}

\noindent Let us take  $s_{1}, s_{2} \geq t_{\varepsilon}$ and $p_{1}, p_{2} \in [p^{\ast}, 1)$ such that $1-p_{1} < \delta_{\varepsilon} $ and $1-p_{2} < \delta_{\varepsilon} $. The Minkowski inequality  implies that for $s \geq t_{\varepsilon}$ 
\begin{align*}
& \parallel W^{(p_{1}), \phi}_{\varnothing}(s_{1}) - W^{(p_{2}), \phi}_{\varnothing}(s_{2})  \parallel_{2}  \\
& \hspace*{10mm}  \leq \hspace*{2mm} \parallel W^{(p_{1}), \phi}_{\varnothing}(s_{1}) - m_{\infty}^{(p_{1}), \phi}W^{(p_{1}), \psi}_{\varnothing}(\infty)   \parallel_{2}  + \parallel  W^{(p_{1}), \phi}_{\varnothing}(s)  - m_{\infty}^{(p_{1}), \phi}W^{(p_{1}), \psi}_{\varnothing}(\infty) \parallel_{2} \\
&  \hspace*{20mm} + \parallel  W^{(p_{1}), \phi}_{\varnothing}(s)  - W^{ \phi}(s)  \parallel_{2} + \parallel  W^{(p_{2}), \phi}_{\varnothing}(s)  - W^{ \phi}(s)  \parallel_{2} \\
& \hspace*{20mm} +  \hspace*{2mm} \parallel W^{(p_{2}), \phi}_{\varnothing}(s_{2}) - m_{\infty}^{(p_{2}), \phi}W^{(p_{2}), \psi}_{\varnothing}(\infty)   \parallel_{2}  + \parallel  W^{(p_{2}), \phi}_{\varnothing}(s)  - m_{\infty}^{(p_{2}), \phi}W^{(p_{2}), \psi}_{\varnothing}(\infty) \parallel_{2}. 
\end{align*}

\noindent By (\ref{eq14}) and (\ref{eq21}), we deduce that $\parallel W^{(p_{1}), \phi}_{\varnothing}(s_{1}) - W^{(p_{2}), \phi}_{\varnothing}(s_{2})  \parallel_{2} \, \, \leq \varepsilon$. Thus, $W^{(p), \phi}_{\varnothing}(t)$ is $L_{2}(\mathbb{P})$-Cauchy. Therefore, the claim in (\ref{eq19}) follows from the well-known completeness of $L_{2}(\mathbb{P})$; see \cite[Theorem 6.14]{Bartle1995}.

Finally, we show that the claim in Lemma \ref{lemma2} holds. Notice that (\ref{eq19}) implies that there exists a square integrable variable $\overline{W}$ such that  $\lim_{p \rightarrow 1, t \rightarrow \infty} W_{\varnothing}^{(p), \phi}(t) = \overline{W}$ in $L_{2}(\mathbb{P})$. Thus, it is enough to show that $\overline{W} = m_{\infty}^{\phi} W^{\psi}(\infty)$, where $\psi$ is defined in (\ref{eq10}) with $p\equiv1$. In this direction, for $t \geq 0$, recall that (\ref{eq13}) shows that $\lim_{p \rightarrow 1} W_{\varnothing}^{(p), \phi}(t) = W^{\phi}(t)$ in $L_{2}(\mathbb{P})$. Furthermore, Lemma \ref{lemma1} implies that 
\begin{eqnarray} \label{eq24}
\lim_{t \rightarrow \infty} \lim_{p \rightarrow 1} W_{\varnothing}^{(p), \phi}(t) = \lim_{t \rightarrow \infty} W^{\phi}(t) =  m_{\infty}^{\phi} W^{\psi}(\infty), \hspace*{5mm} \text{in} \, \, L_{2}(\mathbb{P}).
\end{eqnarray}

\noindent On the other hand, (\ref{eq19}) implies that for $\varepsilon > 0$ there are $\delta_{\varepsilon}, t_{\varepsilon} > 0$ such that for $1-p < \delta_{\varepsilon}$ (for $p \in [p^{\ast}, 1)$) and $s \geq t_{\varepsilon}$ we have that $\parallel W^{(p), \phi}_{\varnothing}(s) - \overline{W} \parallel_{2} \, \, \leq \varepsilon$. Hence (\ref{eq13}) and an application of the dominated convergence theorem allow us to conclude that
\begin{eqnarray*}
\lim_{p \rightarrow 1}  \parallel W^{(p), \phi}_{\varnothing}(s) - \overline{W} \parallel_{2} \, \, = \parallel W^{\phi}(s) - \overline{W} \parallel_{2} \, \, \leq \varepsilon,
\end{eqnarray*}

\noindent i.e.\ $\lim_{t \rightarrow \infty} \lim_{p \rightarrow 1} W_{\varnothing}^{(p), \phi}(t) = \overline{W}$ in $L_{2}(\mathbb{P})$ which combined with (\ref{eq24}) concludes the proof. 
\end{proof}

\paragraph{Acknowledgements.} 
This work was started when I was member of the Institut f\"ur Mathematische Stochastik of Georg-August-Universit\"at G\"ottingen and was supported by the DFG-SPP Priority Programme 1590, {\sl Probabilistic Structures in Evolution}. I would like to thank Juan Carlos Pardo for his comments on an earlier draft of this manuscript. I am very grateful to the referee, whose extremely careful reading and helpful comments led to several improvements in the exposition of this paper. 


\providecommand{\bysame}{\leavevmode\hbox to3em{\hrulefill}\thinspace}
\providecommand{\MR}{\relax\ifhmode\unskip\space\fi MR }
\providecommand{\MRhref}[2]{%
  \href{http://www.ams.org/mathscinet-getitem?mr=#1}{#2}
}
\providecommand{\href}[2]{#2}

\end{document}